\documentclass[12pt]{amsart}
\usepackage{amssymb, amsmath, amsthm}
\usepackage{latexsym, amsmath, bm, amssymb, longtable, booktabs,amscd,microtype,booktabs,cases}
\usepackage[square,numbers]{natbib}
\usepackage{graphicx}
\usepackage[dvipsnames]{xcolor}
\usepackage{caption}
\usepackage{dsfont}
\usepackage[all]{xy}
\usepackage{enumerate}
\usepackage{amsthm}
\usepackage{multirow}
\usepackage{tikz}
\usetikzlibrary{matrix,arrows,decorations.pathmorphing}
\usepackage{collectbox}
\usepackage{enumerate}
\usepackage[colorinlistoftodos,prependcaption, textsize=tiny]{todonotes}
\reversemarginpar
\usepackage{amsmath}
\usepackage{enumitem}
\usepackage{hyperref}
\usepackage{longtable}
\theoremstyle{plain}
\newtheorem{theorem}{Theorem}
\newtheorem{lemma}{Lemma}

\newtheorem{proposition}{Proposition}[section]

\theoremstyle{definition}

\newtheorem{definition}{Definition}[section]

\theoremstyle{remark}
\newtheorem{remark}{Remark}[section]

\newtheorem{example}{Example}[section]

\newcommand{\un}{\underline}

\newcommand{\Z}{\mathbb{Z}}
\newcommand{\la}{\lambda}
\newcommand{\al}{\alpha}
\newcommand{\be}{\beta}

\newcommand{\de}{\delta}
\newcommand{\De}{\Delta}
\newcommand{\C}{\mathbb{C}}
\newcommand{\N}{\mathbb{N}}
\newcommand{\bigo}{\bigoplus}
\newcommand{\ot}{\otimes}
\newcommand{\mf}{\mathfrak}
\newcommand{\dis}{\displaystyle}
\newcommand{\mcal}{\mathcal}
\newcommand{\op}{\oplus}
\newcommand{\til}{\tilde}
\newcommand{\wtil}{\widetilde}

\newcommand{\bs}{\boldsymbol}
\newcommand{\mbf}{\mathbf}
\numberwithin{equation}{section}
\numberwithin{lemma}{section}
\numberwithin{theorem}{section}

\usepackage{graphicx}

\begin{document} 
\title[]{Skew Symmetric Extended Affine Lie algebras}
\author[]{S. Eswara Rao and Priyanshu Chakraborty}
\address{S. Eswara Rao: School of Mathematics, Tata Institute of Fundamental Research, Homi Bhaba Road, Colaba, Mumbai 400005, India.}
\email{ senapati@math.tifr.res.in, sena98672@gmail.com.}
\address{Priyanshu Chakraborty: School of Mathematical Sciences,
	East China Normal University,
	Shanghai, 200241, China.}

\email{    priyanshuc437@.gmail.com}

\maketitle

\begin{abstract} 
For any skew symmetric matrix over complex numbers, we introduce an EALA and it is called Skew Symmetric Extended Affine Lie Algebra (SSEALA). This way we get a large class of EALAs and most often they are non-isomorphic. In this paper we study irreducible integrable modules for SSEALA with finite dimensional weight spaces. We classify all such modules in the level zero case with non degenerate skew symmetric matrix.

\end{abstract}

\bigskip
\noindent
{\bf{Notation:}} 
\begin{itemize}
\item Throughout this paper we will work over the base field $\C$.
\item Let $\C, \mathbb{R}, \Z$ denote the set of complex numbers, set of real numbers and set of integers respectively. 
\item Let $\Z_+$ and $\mathbb N $ denote set of non-negative integers and set of positive integers respectively. For $n \in \N$, $\C^n = \{(x_1, \ldots , x_n) : x_i \in \C, 1 \leq i \leq n\}$ and
  $\mathbb{R}^n, \Z^n, \N^{n}$ and $\Z_{+}^{n}$ are defined similarly. 
 \item  Elements of $\C^n, \mathbb{R}^n$ and $\Z^n$ are written in boldface.
 
 \item For any Lie algebra $\mathfrak{g}$, $U(\mathfrak{g})$ will denote universal enveloping algebra of $\mathfrak{g}$.
 \item Let $(\cdot  | \cdot)$ denote the standard inner product on $\C^n$.
 \item For any matrix $B$, $B^T$ denote the transpose matrix of $B.$ 
 \item For all $n \in \N$,  $GL(n, \C)$ denote the set of all $n \times n$ invertible matrices over $\C$.  
\end{itemize}

\section{Introduction} 
In this paper we introduce a class of Extended Affine Lie Algebras (EALA) which we call Skew Symmetric Extended Affine Lie Algebras ( for short SSEALA). For any skew symmetric matrix $B$ $(B^T=-B)$ over complex numbers we define an EALA $\tau_B$ which we call  Skew Symmetric Extended Affine Lie Algebra (SSEALA). The first author in \cite{RH} introduced Hamiltonian Extended Affine Lie Algebra (HEALA) and a Contact Extended Affine Lie Algebra (KEALA). They both are SSEALA and the underlying skew symmetric matrix is non-degenerate for HEALA (see Example \ref{e3.1}) and degenerate for KEALA (see Example \ref{e3.2}).\\

EALAs have been extensively studied in the last two decades, see \cite{AABGP,AG,N1,N2,ABFP} and references therein. There authors mainly study the structure of an EALA and able to give a definite shape to the core of EALA (see subsection \ref{co2.1} for the definition of core of an EALA and in particular for $\tau_B$ core is $\mf g \ot \mcal A \op \wtil {\mcal Z}$, for more details see the definition of $\tau_B$ from Section 3). Later several authors studied representations of EALAs, see (\cite{RH,RSB,TB,CLT1,CLT2,PR,SP1,SP2,YY}). In fact the first author of the current paper has classified irreducible integrable modules for HEALA (both level zero and non-zero level) in \cite{RH}.  \\

In this paper after introducing SSEALA we attempt to study integrable modules for these EALAs assuming the skew symmetric matrix to be non-degenerate. In the degenerate case the problem is more challenging and our methods do not apply.\\


We will now give details of each section. In Section 2 we recall the definitions of a toroidal Lie algebra, a full toroidal Lie Algebra, a toroidal Extended Affine Lie Algebra and a Hamiltonian Extended Affine Lie Algebra. We note that the root systems of all these algebras are the same but with different multiplicities. We also recall the definition of EALA and provide some examples. \\
In Section 3, for any skew-symmetric matrix $B$, we define the Skew Symmetric Extended Affine
Lie Algebra (SSEALA) $\tau_B$, which is the main object of study in this paper. Basic properties of these
algebras are established in Proposition 3.1. In Examples 3.1 and 3.2 we show that HEALA and KEALA
are SSEALAs by exhibiting the corresponding skew symmetric matrices. Since it is well-known (see for example \cite{Sam P}, Lemma 3.1) that
for every non-degenerate skew symmetric $B$ there exists an invertible matrix $A$ such that $A^TBA = J_m,$
where $J_m=   \begin{pmatrix}
	0 & I_{m \times m} \\
	-I_{m \times m} & 0 \\ 
\end{pmatrix} ,$
a natural question arises whether $\tau_B$ with non-degenerate $B$ is isomorphic to $\tau_{J_m}$. Yet
this turns out to be the case only if $A \in GL(2m, \Z)$.

Section 4 is our main technical section. Let us denote $\mcal A= \C[t_1^{\pm{1}},t_2^{\pm{1}},\dots, t_N^{\pm{1}}].$ All EALAs we are working with contains a derivation algebra $H_{*}$ which is a subalgebra of derivation algebra of $\mcal A,$ denoted by $Der \mcal A.$ Such an algebra for a
SSEALA $\tau_B$ will be denoted by $H_B$. As seen previously in case of Toroidal Extended Affine Lie Algebra (TEALA) in \cite{RSB} and Hamiltonian Extended Affine Lie Algebra (HEALA) in \cite{RH}, the problem of classifying irreducible integrable modules for $\tau_B$  reduces to that of classifying certain Jet modules over $H_{B'} \ltimes \mcal A$, where $B'$ is obtained from $B$ by removing some rows and the corresponding columns of $B.$ 
Now we will state the main theorem of Section 4. We need some notation for that. We assume that B is non degenerate so $N= 2m$ (see Lemma 3.1). Set
$$H_B=span\{D(B{\mbf r},{\mbf r}): {\mbf r} \in \Z^N \setminus \{\mbf 0\}\}\, \text{and} \,\, \mcal D=span\{D(\mbf u,\mbf 0): {\mbf u} \in \C^N \}, $$
see equation (2.1) for the definition of $D(\mbf u,\mbf r)$,  $ \forall \,\,\mbf r \in \Z^N, \, \mbf u \in \C^N$ and set
$$ \widetilde {H_B}=H_B \op \mcal D.  $$
 Note that $\widetilde {H_B}$ is a subalgebra of $Der \mcal A$ and hence the action of $Der \mcal A$ on $\mcal A$ induces an action of $\widetilde {H_B}$ on $\mcal A$. This action defines a Lie algebra structure on $\wtil{H_B} \ltimes \mcal A$. The following theorem states about irreducible jet module (see Remark 4.1 for definition of jet modules) structures for $\wtil{H_B} $. \\

\begin{theorem}\label{t1.1}
	
	(1) Let $V$ be an irreducible jet module for $\wtil{H_B} $ with finite dimensional weight spaces. Then $V \simeq V_{\mbf 0} \otimes \mcal A$ as $\wtil{H_B} $ module, where $V_{\mbf 0}$ is an irreducible module for a reductive Lie algebra whose semi simple part is $\mf{sp}_{2m}$.

\end{theorem}

For more details on Theorem 1.1, see Theorem 4.1 and Remark 4.1.

In Section 5 we will classify irreducible integrable modules for $\tau_B$ (see construction of $\tau_B$ in section 3) with finite dimensional weight spaces on the assumption that B is non degenerate and the level is zero ( SSEALA comes with a finite dimensional center. A module for SSEALA is called level zero if the finite dimensional center acts trivially. Otherwise it is called non-zero level module). In Proposition 5.1 we proved that when the level is zero the center of the core (see section 2.1) acts trivially on an irreducible integrable module with finite dimensional weight spaces. One should observe that Lie algebra $\tau_B$ obtained from the multiloop algebra of finite dimensional simple Lie algebras and the following theorem states how irreducible modules for finite dimensional simple Lie algebras occurs in the irreducible integrable modules for $\tau_B$.
\begin{theorem}
	Let $V$ be an irreducible integrable level zero module for $\tau_B$ having finite dimensional weight spaces with respect to $\wtil {\mf h}$ (see Section 2 for definition of $\wtil {\mf h}$). We also assume that $B$ is non degenerate and the core of $\tau_B$ acts non-trivially on $V$. Then there exists a finite dimensional irreducible module $V(\mu)$ for $\mf g$ and a finite dimensional irreducible module $P$ for $\mf{sp}_{2m}$ such that $V\simeq V(\mu)\ot P \ot \mcal A$. The action of $\tau_B$ on $ V(\mu)\ot P \ot \mcal A$ given by following:\\
	$X({\bf r}).(v\ot w\ot t^{\bf s})=X.v\ot w\ot t^{\bf r+s}$, for all $X \in \mf g$, $\bf r, \, s$ $ \in \Z^N,$ $v \in V(\mu), w \in P$,
	action of $\wtil {H_B}$ is given by (4.3), (4.4) and center of core acts trivially on $V$.
\end{theorem}

In Section 6 we classify irreducible integrable modules for KEALA in the non-zero level case. The level zero case is still open. The motivation for their study is the following. 
 Recall that Lie algebras of Cartan type which arise in the study of vector fields on manifolds
 belong to four infinite series, namely the general, special, Hamiltonian and contact, denoted respectively
 by $W_N , S_N , H_N$ and $K_N $.  We have EALAs for type S (TEALA, see Example 2.2), and type H (HEALA, see Example 3.1). For type W we have full toroidal Lie algebra which will fall short of EALA as it satisfy all the axioms of EALA except the existence of non-degenerate invariant bilinear form. This problem can be rectified by taking a suitable subalgebra of a derivation algebra and this is TEALA (S type).\\
 Rudakov \cite{RA} has defined all four types of Cartan type Lie algebras but over a polynomial algebra of several variables. All types can be generalized over Laurent polynomial algebras except type $K$ as noted by J. Talboom, \cite{TH}. By closely looking at type $K$ of Rudakov, First author proposed in \cite{RH} KEALA ( Example 3.2) for $K$ type in the Laurent polynomial algebra case. We will now state a classification result for KEALA in the non zero level case. The corresponding skew symmetric matrix is $\bar J$ (see Example 3.2) which is of co-rank 1, i.e $N= 2 m + 1$. We choose $B'$ by removing a suitable row and the corresponding column so that $B'=J_m$ which is non degenerate. Then we apply Theorem \ref{t1.1} to get the classification theorem for in this case.
 \begin{theorem}
 	Let $V$ be an irreducible integrable non zero level module for $\tau_{\bar J}$ with finite dimensional weight spaces with respect to $\wtil {\mf h}$ and $\mf h$ acts non-trivially on $V$. Then after twisting the module by an automorphism of $\tau_{\bar J}$, we have $V \simeq L(P'\ot \mcal A_{2m})$, where $P'$ is a finite dimensional irreducible $\mf{sp}_{2m} $ module. Here $L(P'\ot \mcal A_{2m})$ is the unique irreducible quotient of an induced module (see subsection 6.1 for the construction of $L(P'\ot \mcal A_{2m})$).
 \end{theorem}

\section{Toroidal Lie algebras and Full toroidal Lie algebras}
Fix a positive integer $N.$ Let $\mcal A= \mcal A_{N}= \C[t_1^{\pm{1}},t_2^{\pm{1}},\dots, t_N^{\pm{1}}]$ denote the Laurent polynomial ring in $N$ commuting variables. For ${\bf{r}}=(r_1,r_2, \dots ,r_N) \in \Z^N$ denote $t^{\bf{r}} = t_1^{r_1}t_2^{r_2} \dots t_N^{r_N} \in \mcal A.$\\
Let $\mf{g}$ be a finite dimensional simple Lie algebra and $\mf{h}$ be a Cartan subalgebra of $\mf{g}$. It is well known that $\mf{g}$ has a root space decomposition with respect to $\mathfrak{h}$ given by
$\mf{g}=\mf{h} \op \dis{ \bigoplus_{\alpha \in \Delta}} \mf{g}_\alpha $, where $\De$ is the corresponding finite root system. Let $<.,.>$ be a symmetric, non-degenerate, invariant bilinear form on $\mf{g}$.  \\

Let us consider the multi-loop algebra $\mathfrak{ g} \otimes  \mcal A$ with the usual Lie bracket.
Let \begin{center}
    
$\Omega _{\mcal A} =span\{t^{ {\mbf r}}K_i: 1\leq i\leq  N, \mbf{r}  \in \Z^N  \},$ where $K_i=t_i^{-1}dt_i$, for $1 \leq i \leq N$.\end{center} 
We define a $\Z^N$-grading on $\Omega _{\mcal A}$ by setting deg ${t^{\mbf r}K_i=\mbf r}$, $ 1 \leq  i \leq N$. Clearly, each homogeneous component of $\Omega_{\mcal A}$  is $N$-dimensional. Let $d_{\mcal A}$ be the subspace of $\Omega_{\mcal A}$ defined by $ d_{\mcal A}= span\{\displaystyle{\sum_{i=1}^{N}}  r_i  t^{\mbf{r} } K_i:   {\mbf{r}}\in \mathbb{Z}^N   \}$ and set ${\mcal Z}=\Omega_{\mcal A}/d_{\mcal A}$. Then $\mcal Z$ is $\Z^N$-graded, more precisely each non-zero graded components of $\mcal Z$ are $N-1$ dimensional and zeroth graded component is $N$ dimensional. Let $K({\mbf{u},\mbf{r}})=\dis{\sum_{i=1}^{N}{u_i{t^{\mbf{r}}} K_i}}$ for ${\mbf{u}} \in \C^ N, {\mbf r} \in \Z^ N$. Now we define toroidal Lie algebra. As a vector space toroidal Lie algebra is given by 
$$ \tau = \mf{g} \ot \mcal A \oplus \mcal Z \op \mcal D ,  $$ where $\mcal D$ is the degree derivation space spanned by $\{d_1,d_2,\dots, d_N \}$. Before defining the Lie bracket of the toroidal Lie algebra we fix some notations for this paper. For convenience we set $X({\mbf {r}})= X \ot t^{\mbf r}$ and $\mf{g}({\mbf r})= \mf{g}\ot \C t^{\mbf r}$, $\mbf r \in \Z^N$ and $X \in \mf{g}.$ Lie brackets of the toroidal Lie algebra are given by 
\begin{enumerate}
    \item[A1.] $[X({\mbf r}), Y({\mbf s})]=[X,Y]({\mbf{r+s}}) + <X,Y>K({\mbf{r,r+s}}).$
    \item[A2.] $\mcal Z$ is central in $\mf{g} \ot \mcal A.$
    \item[A3.] $[d_i, X({\mbf r})]=r_iX({\mbf r}),$ $[d_i, K({\mbf u, \mbf r})]=r_iK({\mbf u,\mbf r})$, $[d_i,d_j]=0,$\\ for all ${\mbf r, \mbf s} \in \Z^N, {\mbf u} \in \C^N, X, Y \in \mf{g}, 1 \leq i, j \leq N.$ 
    \end{enumerate}

Let $\wtil{\mf h}=\mf h \op \mcal Z_0 \op \mcal D$, where $\mcal Z_0=$ span $\{K_1, K_2, \dots, K_N \}$. This $\wtil {\mf h}$ plays the role of Cartan subalgebra for toroidal Lie algebra $\tau.$ One should observe that the toroidal Lie algebra is the $N$ variable generalization of affine Kac-Moody Lie algebra (see \cite{VK} for affine Kac Moody Lie algebras). For more details on toroidal Lie algebras one can see \cite{MRY}.\\
Let $Der \mcal A$ be the derivation algebra of $\mcal A.$ It is well known that $Der \mcal A$ has a basis given by the vectors $\{t^{\mbf r}d_i: {\mbf  r} \in \Z^N, 1 \leq i \leq N \},$ where $d_i=t_i \frac{d}{d t_i}$. Let 

\begin{align}
	D({\mbf u, \mbf r})= \dis{\sum_{i=1}^{N}}u_it^{\mbf r}d_i,
	\end{align} where ${\mbf u}=(u_1,u_2,\dots, u_N) \in \C^N$ and ${\mbf r} \in \Z^N.$ Then $Der \mcal A$ forms a Lie algebra with respect to the bracket operation:
\begin{align} 
	\left[ D({\mbf u,\mbf r}),D({\mbf v,\mbf s}) \right]=D({\mbf w,\mbf r+ \mbf s}),
\end{align}
where $ {\mbf w}= ({\mbf u}| {\mbf s}){\mbf v} -({\mbf v}| {\mbf r}){\mbf u} $ for all ${\mbf r, \mbf s} \in \mathbb{Z}^N$, ${\mbf u, \mbf v} \in \mathbb{C}^N,$ $(\mbf u| \mbf v)=$ $\dis{\sum_{i=1}^{N}}u_iv_i $, for ${\mbf u}=(u_1,\cdots,u_N)$, ${\mbf v}=(v_1,\cdots,v_N)$.\\
It is known that $Der \mcal A$ admits an abelian extension on $\mcal Z$ by the following actions:
\begin{align}
    [ D({\mbf u,\mbf r}),D({\mbf v,\mbf s})]=D({\mbf w,\mbf r+ \mbf s})-({\mbf u}| {\mbf s})({\mbf v}| {\mbf r})K({\mbf{r,r+s}}),
\end{align}
\begin{align}\label{a2.3}
	[ D({\mbf u,\mbf r}),K({\mbf v,\mbf s})]=({\mbf u| \mbf s})K({\mbf v,\mbf r+ \mbf s})+({\mbf u}| {\mbf v})K({\mbf{r,r+s}}),
\end{align}
for more details see \cite{RM1}.
Moreover it is known from \cite{RSS} that $Der \mcal A$ has no non trivial central extension for $N\geq 2.$ For $N=1$ above defined abelian extension becomes central extension. Now we are prepared to define the full toroidal Lie algebra. As a vector space, the full toroidal Lie algebra $\til \tau= \til \tau_{\mf g}$ corresponding to $\mf g$ is equal to,
$$ \til \tau =  \mf{g} \ot \mcal A \oplus \mcal Z \op Der \mcal A. $$
The Lie bracket on $\til \tau$ is given by A1, A2, (2.2), (2.3) and the following:
\begin{align}
    [ D({\mbf u,\mbf r}),X({\mbf s})]=({\mbf u| \mbf s})X({\mbf r+\mbf s})
\end{align}
for all ${\mbf r, \mbf s} \in \mathbb{Z}^N$, ${\mbf u} \in \mathbb{C}^N.$ An interested reader can see \cite{Y1,RSM} for representations of full toroidal Lie algebras.\\
 Let $GL(N, \Z)$ be the groups of integral matrices with determinant $\pm 1$. Then there is a natural action of $GL(N,\Z)$ on $\Z^N.$ Let us denote this action by $F {\mbf r}$ for $F \in GL(N, \Z)$ and ${\mbf r} \in \Z^N.$ Given $F \in GL(N,\Z) ,$ the assignments:
\begin{align}\label{a2.5}
    F.X({\bf r})=X(F{\bf r}), \hspace{.5cm} F.K({\bf u,\bf r})=K(F{\bf u}, F{\bf r}), \hspace{.5cm } F.D({\bf u,\bf r})=D((F^T)^{-1}{\bf u}, F{\bf r}).
\end{align}
define an automorphism of $\til \tau $ which we denote by $\Phi_F$ (see \cite{RJ} for details).\\
We now define roots and coroots of $\til \tau.$ Note that root system of $\til \tau$ is same as the root system of toroidal Lie algebra. Therefore we just recall root system for $\tau$ from \cite{RT}. For $1 \leq i \leq N$ define $\delta_i, \omega  _i \in {\wtil{\mf h}}^{\ast} $ (${\wtil{\mf h}}^{\ast}$ represents the dual of $\wtil{\mf h}$ ) such that
\begin{center}

$ \delta _i (\mathfrak{ h}) =0, \hspace{.5cm} \delta _i (K_j)=0 , \hspace{.5cm} \delta _i (d_j)= \delta_{ij} ,  \hspace{.5cm} 1 \leq j \leq N; $\\
 
$ \omega_i (\mathfrak{ h}) =0  , \hspace{.5cm} \omega _i (K_j)= \delta _{ij} , \hspace{.5cm} \omega_i (d_j)=0 ,  \hspace{.5cm} 1 \leq j \leq N. $ 
\end{center}
For $  {\bf {m}} =(m_1,...,m_N) \in \mathbb{Z}^N $, set  $ \delta_ {\bf m}= $ $ \displaystyle \sum_{i=1}^{N} m_i \delta_i $. Let $\pi=$ $ \lbrace \alpha_1 , \alpha_2,..., \alpha_d \rbrace$ be a set of simple roots of  ${\Delta} $ and $\pi^\vee=\{\alpha_1^\vee, \alpha_2^\vee,\dots, \alpha_d^\vee\}$ be the set of corresponding coroots of $\De$, where $d=$ dim $\mf{h}$. Then the set of vectors $\{ \alpha_1 , \alpha_2,..., \alpha_d,\de_1,\de_2,\dots, \de_N, \omega_1,\omega_2,\dots, \omega_N\}$ forms a basis for ${\wtil {\mf{h}}}^*$. Now we extend the bilinear form of $\mf h^*$ to a symmetric non-degenerate bilinear form on ${\wtil {\mf{h}}}^*$ by defining:
\begin{center}
    $	 <\alpha , \delta_k>=<\alpha, \omega_k>=0 $,\hspace{.2cm} for all $ \alpha\in \De, 1 \leq k \leq N ,$\\
    	$ < \delta_k,\delta_p>=<\omega_k, \omega_p>=0 ,  \hspace{.2cm}<\omega_k,\delta_p>=\delta_{kp}, \hspace{.2cm  } 1 \leq k,p \leq N . $
\end{center}
Again $\wtil{\mf h}$ has a basis $\{\alpha_1^\vee, \alpha_2^\vee,\dots, \alpha_d^\vee,K_1,K_2, \dots, K_N,d_1, d_2,\dots, d_N     \}$ with bilinear form defined by:
\begin{center}
    $<h,K_i>=<h,d_i>=0,$ for all $h \in \mf h$ and $1 \leq i \leq N.$\\
    $<d_i,d_j>=<K_i,K_j>=0$ and $<d_i,K_j>=\de_{ij}$ for $1\leq i,j \leq N.$
\end{center}
It is clear that dim $\mf {\wtil h}$ = dim $\mf{ \wtil h}^*=2N+d$. Let $\De^{re}=\{\alpha + \de_{\bf r}: \alpha\in \De , {\bf r} \in \Z^N\}$ and $\De^{im}=\{ \de_{\bf r}: {\bf r} \in \Z^N\}.$ Then $\wtil{\De}=\De^{re} \cup \De^{im}$ is the root system of $\til \tau$ with respect to the Cartan subalgebra $\til {\mf h}.$ \\
Let $\bar{\la}$ denote the restriction of $\la \in \wtil{\mf h}^*$ to $\mf h.$
Again any $\mu \in \mf h^* $ can be extend to $\wtil{\mf h}^*$ by defining $\mu(d_i)=\mu(K_i)=0$ for $1 \leq i \leq N.$ Therefore any $\la \in \wtil{\mf h}^*$ can be expressed, uniquely, as $\la = \bar{\la} + \dis{\sum_{i=1}^{N}g_i\de_i}+\dis{\sum_{i=1}^{N}s_i\omega_i},$ see \cite{RT} for reference.\\
We also define $\alpha_{d+j}=-\beta + \de_j$ for $1 \leq j \leq N$, where $\beta \in \De$ is the highest root $\mf g.$ It is easy to see that the set of vectors $\{\alpha_1, \dots, \alpha_d, \alpha_{d+1},\dots, \alpha_{d+N}, \omega_1,\dots,\omega_N\}$ forms another basis of $\wtil{\mf h}^*.$ For $\gamma=\alpha  + \delta_{\mbf m} \in \Delta^{re} $, define the corresponding coroot by
$$  \gamma^{\vee} = \alpha^{\vee} + \frac{2}{<\alpha, \alpha>}{\displaystyle \sum_{i=1}^{N}}m_iK_i. $$
For a real root $\gamma \in \Delta^{re},$ define a reflection $r_\gamma$ on $ {\wtil{\mf h}}^*$ by $$ r_{\gamma}(\lambda)= \lambda - \lambda(\gamma^{\vee}) \gamma , \hspace{.2cm}\lambda \in {\wtil{\mf h}}^* .$$
Let $ \wtil W $ be the Weyl group generated by the reflections corresponding to real roots of $\wtil \Delta $. 
\begin{definition}\label{d321}
	A module $V$ for $\til \tau$ (or over $\tau$ ) is said to be integrable if $V$ satisfies the following properties: 
 \begin{enumerate}
     \item $V $ can be decomposed as $ V=\displaystyle { \bigoplus_{\lambda \in {\til{\mf h}}^*}}{V_{\lambda}}$, where $ V_{\lambda} =\{v \in V: hv=\lambda(h)v , \, \forall \, h \in {\wtil{\mf h}} \}.$
	
	\item $X_{\alpha}({\mbf m})$ acts locally nilpotently on $V$ for all $\alpha \in \De, X_\alpha \in {\mf g}_\alpha, {\mbf m} \in \Z^N .$

 \end{enumerate}
\end{definition}
Let $V$ be an integrable module for $\til \tau $ (or over $\tau$) with finite dimensional weight spaces, i.e dim $V_\lambda < \infty$ for all $\la \in {\wtil{\mf h}}^*.$  Let $P(V)= \{ \la \in {\wtil{\mf h}}^*: V_\la \neq 0\}$ be denote the set of all weights of $V$. Then the following are very standard ( see \cite{RT} ).
\begin{lemma}\label{l2.1}

    (1) $P(V)$ is $\widetilde W$ invariant.\\
    (2) dim $V_\la =$ dim $V_{w\la}$ for all $w \in \widetilde W, \la \in P(V).$\\
    (3) If $\alpha \in \De^{re}, \la \in P(V)$ and $\la(\alpha^{\vee}) > 0,$ then $\la -\alpha \in P(V)$

\end{lemma}

\begin{definition} Extended affine Lie algebra (EALA) : Let $L$ be a Lie algebra such that:\\
\begin{enumerate}
    \item[EA1.] $L$ is endowed with a non-degenerate symmetric bilinear form $(.,.)$ which is invariant (i.e $([x,y],z)=(x,[y,z]), $ for all $x,y,z \in L).$
    \item[EA2.] $L$ possesses a non-trivial finite dimensional self-centralizing ad-diagonalizable  abelian subalgebra $H$. 
\end{enumerate}
    To complete the definition of EALA, we need three more axioms. We observe some consequence of EA2 which will be useful to define other three axioms. Note that by property EA2 we have 
    \begin{center}
    $ L=\displaystyle { \bigoplus_{\alpha \in H^*}}{L_{\alpha}}$,
where  $L_{\alpha} =\{x \in L: [h,x]=\alpha(h)x , \forall h \in H \}.$
\end{center}
Let $R=\{\al \in H^*: L_\al  0 \},$ this set is called root system of $L$ with respect to $H$ (observe that this notion of root system is different from usual one, in the usual case root system do not contain $0$, for reference see \cite {AG}). Note that $0 \in R$ and $\al, \be \in R$ with $ \al + \be \neq 0 \implies (L_\al, L_\be) = 0.$ Hence we have $R =-R$. The form restricted to $H$ is non-degenerate, thus we can induce a form on $H^*$. Let
\begin{center}

$R^\times = \{ \al \in R: (\al , \al) \neq 0 \}  ,$ \hspace{.5cm } $ R^0 = \{ \al \in R: (\al , \al)=0 \} .$

\end{center}
The elements of $R^\times$( respectively $R^0)$ are called non-isotropic (respectively isotropic) roots. It is clear that $R=R^\times \cup R^0.$
\begin{enumerate}
    \item[EA3.] If $x_\al \in L_\al$ for $\al \in R^\times $, then $adx_\al$ act locally nilpotently on $L.$
    \item[EA4.] $R$ is a discrete subset of $H^*$.
    \item[EA5.] $R$ is an irreducible root system. This means that\\
    (i) If $R^\times =R_1\cup R_2$ and $(R_1,R_2)=0$, then either $R_1=0$ or $R_2=0$.      \\
    (ii) If $\sigma \in R^0$, then there exists a $\alpha \in R^\times$ such that $\sigma +\al \in R.$
\end{enumerate}
A Lie algebra satisfying EA1 to EA5 is called an EALA.\\
\subsection[]{Core of EALA }\label{co2.1}The core of the EALA $L$ is defined as the subalgebra generated by $\dis{\bigcup_{\al \in R^\times}}L_\al$. It should be mentioned that the core of $L$ is an ideal of $L$.\\ Extensive research has been done on EALAs and their representation theory. One can observe that the full toroidal Lie algebra is not an EALA as it does not satisfy (EA1). Now we discuss some known examples of EALAs. 
\end{definition}
\begin{example}
	Consider the Lie algebra  $\mf{g} \ot \mcal A \op \dis{\bigoplus_{1 \leq i \leq N}}\C K_i \op \dis{\bigoplus_{1 \leq i \leq N}}\C d_i $. Lie bracket is given by 
	\begin{center}
		 $[X({\mbf r}), Y({\mbf s})]=[X,Y]({\mbf{r+s}}) + \de_{\mbf{r+s},\mbf 0}<X,Y>\dis{\sum_{i=1}^{N}}\C K_i, $\\
		 $[K_i,K_j]=[d_i,d_j]=[K_i,d_j]=0$,\\
		 $[d_i,X({\mbf r})]=r_iX({\mbf r}),$\\
		 and $K_i$ are central, for all $X, Y \in \mf g, {\mbf r, \mbf s} \in \Z^N, 1 \leq i,j \leq N$.
	\end{center}
It is the minimal EALA coming from the extension of a multiloop algebra. Irreducible modules for this Lie algebra have been studied in \cite{SP}, also see \cite{R05}.
\end{example}
\begin{example}
	A certain subalgebra of $\til \tau$ forms an EALA. Define
	$S_N=span\{ D({\bf u, \bf r}):({\bf u| \bf r})=0, {\bf u} \in \C^N, {\bf r} \in \Z^N \},$ this is a subalgebra of $Der \mcal A$. It is known to be a simple Lie algebra of $S$ type. For represenation on
	$S_N$ one can see \cite{BT}.
	 Let 
	$$ \tau(S_N)=  \mf{g} \ot \mcal A \oplus \mcal Z \op S_N.$$
	Define a bilinear form on $\tau(S_N)$ by 
	\begin{center}

		$(X({\bf r}),Y({\bf s}))=<X,Y>\de_{\bf r+\bf s,0},$ for all $X,Y \in \mf g$, $\bf r,s$ \ $\in \Z^N$,\\
	$(D({\bf u, \bf r}),K({\bf v}, {\bf s}))=\de_{\bf r+ \bf s,0}({\bf u}|{\bf v})$ for all $\bf r, \bf s $ $\in \Z^N$, ${\bf u, \bf v}\in \C^N$,\\
	and all other values are zero.
		\end{center}
	It is a standard fact that $\tau(S_N)$ is an EALA with the above bilinear form (see \cite{N2} for general construction of EALAs). It is the largest EALA coming from the extension of a multiloop algebra. There is a analogous notion of twisted version of $\tau(S_N)$ and irreducible integrable module for these Lie algebras have been studied in \cite{RSB,TB,CLT1,CLT2}.
\end{example}
Before going to next example we fix the notation $J_m=   \begin{pmatrix}
	0 & I_{m \times m} \\
	-I_{m \times m} & 0 \\ 
\end{pmatrix} $
for all $m \in \N$, where $I_{m \times m}$ is the $m \times m$ identity matrix.
\begin{example}
	We now discuss another class of EALA subquotient of full toroidal algebras. First, we define
	Hamiltonian Lie algebras (cf. \cite{TH}). Let $N=2m$,  
	   $H_N =span\{D(J_m{\mbf r},{\mbf r}): {\mbf r} \in \Z^N \setminus \{\mbf 0 \} \} $ and $\wtil{H_N}=H_N \op \mcal D$. It is easy to verify that $\wtil{H_N}$ is a Lie algebra with respect to the following Lie bracket\\
	 \begin{center}

	 $[D(J_m{\mbf r},{\mbf r}),D(J_s{\mbf s},{\mbf s})]=(J_m{\mbf r}|{\mbf s})D(J_m({\mbf r +\mbf s}),{\mbf r +\bf s})$ and\\
	  $[D({\mbf u},\mbf 0),D(J_m{\mbf r},{\mbf r}) ]$=$({\mbf u| \mbf r})D(J_m{\mbf r},{\mbf r})$ for all ${\mbf r, \mbf s} \in \Z^N, {\mbf u} \in \C^N$.
	  \end{center}
  The Lie algebra $\wtil{H_N}$ is known as Hamiltonian Lie algebra and it is also called as $H$ type Lie algebra.\\
  Let $\mcal K=span\{ K({\mbf u, \mbf r}) \in \mcal Z: ({\mbf u}| J_m{\mbf r})=0, \, {\mbf u} \in \C^N, \, {\mbf r} \in \Z^N \setminus \{\mbf 0\} \}$. It is trivial to check by (\ref{a2.3}) that $[H_N,\mcal K] \subseteq \mcal K.$ Now consider the Lie algebra, $  \mf{g} \ot \mcal A \oplus \mcal Z \op \wtil{H_N}.$ Note that $\mcal K$ is an ideal in this Lie algebra. Define $\tau(H_N)=  \mf{g} \ot \mcal A \oplus  \mcal Z/\mcal K \op \wtil{H_N}.$ It is proven in \cite{RH} that $\tau(H_N)$ is an EALA, which is called a Hamiltonian EALA (HEALA). Irreducible integrable modules for HEALA have been studied in \cite{RH}.
\end{example}

\section{ Skew symmetric Extended Affine Lie algebras }
In this section we introduce a class of extended affine Lie algebras which we refer to as Skew symmetric extended affine Lie algebras (in short SSEALAs). This class generalizes hamiltonian EALAs described in Example 2.3. For general construction of EALAs one can see the reference \cite{N2}.\\
Let $B$ be a non zero $N\times N$ skew symmetric matrix over $\C$ i.e, $B^T=-B$. 
Set $$H_B=span\{D(B{\mbf r},{\mbf r}): {\mbf r} \in \Z^N \setminus \{\mbf 0\}\}\, \text{and} \,\,  \widetilde {H_B}=H_B \op \mcal D.$$
 It is easy to see that $H_B$ forms a subalgebra of $Der \mcal A$ with the Lie bracket   $$[D(B{\mbf r},{\mbf r}),D(B{\mbf s},{\mbf s})]=(B{\mbf r}|{\mbf s})D(B({\mbf r + \mbf s}), {\mbf r+ \mbf s}), \, \forall  \, {\mbf r, \mbf s } \in \Z^N \setminus \{\mbf 0\}.$$
 In particular, for $B=J_m$ the Lie algebra $H_{J_m}$ coincides with the Lie algebra $H_N$ constructed in Example 2.3.\\
For a skew symmetric matrix $B$ we associate a skew symmetric bilinear form on $\C^N$ by $({\mbf x, \mbf y})_B=(B{\mbf x|\mbf y})=({\mbf x}|B^T{\mbf y})=-({\mbf x}|B{\mbf y})$. We record some properties of $(\mbf x,\mbf y)_B$.
\begin{enumerate}
   \item [(3.1)] $({\mbf r},{\mbf r})_B=(B{\mbf r}|{\mbf r})=0$, for all ${\mbf r} \in \C^N$.
   \item [(3.2)] $({\mbf r},{\mbf s})_B=(B{\mbf r}|{\mbf s})=0$, for all ${\mbf r} \in \C^N,$ ${\mbf s} \in ker B$.
    \item[(3.3)] $({\mbf r},{\mbf s})_B=(B{\mbf r}|{\mbf s})=0$, for all ${\mbf r +\mbf s} \in kerB$. 
\end{enumerate}
Define $\mcal K_B=span\{ K({\mbf u, \mbf r}) \in \mcal Z: (B{\mbf u}|{\mbf r})=0, \, {\mbf u} \in \C^N, \, {\mbf r} \in \Z^N \setminus \{\mbf 0\} \}$. It is trivial to check that $[H_B,\mcal K_B] \subseteq \mcal K_B$  (see (2.3)). Now consider the Lie algebra, $  \mf{g} \ot \mcal A \oplus \mcal Z \op \wtil{H_B}.$ Note that $\mcal K_B$ is an ideal in this Lie algebra. Let us denote $\wtil {\mcal Z} = \mcal Z/\mcal K_B $.  Let us define SSEALA by $\tau_B=  \mf{g} \ot \mcal A \oplus \wtil {\mcal  Z} \op \wtil{H_B}.$ Set $G_B=ker B \cap \Z^N$.

\begin{proposition}\label{p3.1}

		(1)  dim $(\wtil Z)_{\mbf r}=0 $ if $\mbf 0 \neq {\mbf r} \in G_B$.\\
		(2)  dim $(\wtil Z)_{\mbf r}=1 $ if $ {\mbf r} \notin G_B$.\\
		(3) dim $(\wtil Z)_{\mbf 0}=N $.\\
		($4$)  dim $(\wtil H_B)_{\mbf r}=0 $ if $\mbf 0 \neq {\mbf r} \in G_B$.\\
		(5) dim $(\wtil H_B)_{\mbf r}=1 $ if $ {\mbf r} \notin G_B$.\\
		(6) dim $(\wtil H_B)_{\mbf 0}=N $.\\
		(7) $\tau_B$ is an EALA.\\
		(8) The core of $\tau_B$ is $\mf g \ot A \op \wtil Z$.\\
	
\end{proposition}
\begin{proof}
Note that $(1)$ and $(3)-(6)$ follows from definitions. To prove $(2)$, let ${\mbf r} \notin G_B$ and ${\mbf{u}, \mbf{v} }\in \C^N$ such that $(B{\mbf u}|\mbf r)\neq 0 $ and $(B{\mbf v}|\mbf r)\neq 0 $. Now consider
 $$(B(\frac{\mbf u}{(B\mbf u|\mbf r)}-\frac{\mbf v}{(B\mbf v|\mbf r)})|\mbf r)  =0. $$
 This implies that $K({\mbf u, \mbf r})=\la K({\mbf v, \mbf r})$ (mod $\mcal K_B$) for some non zero $\la \in \C$. Therefore $K(B\mbf r,\mbf r) + \mcal K_B$ can be taken as a basis element for $(\wtil Z)_{\mbf r}$ when ${\mbf r} \notin G_B$.\\ To prove $(7)$ we first
 observe that with respect to the Cartan subalgebra $\mf {\til h}$ root space decomposition of $\tau_B$ is given by $\tau_B= \til{\mf h} \op \dis{\bigoplus_{\be \in \til \De}} (\tau_B)_\be$, where 
 $$  (\tau_B)_\be = 
 \begin{cases}
 	\mf g_\al \ot t^{\mbf r}, \, \text{if} \, \, \be=\al+\de_{\mbf r}, \, \al \in \De, {\mbf r }\in \Z^N\\
 	\mf h \ot t^{\mbf r} \op (K(B{\mbf r} ,{\mbf r}) + \mcal K_B) \op  D(B{\mbf r} ,{\mbf r}), \, \text{if} \, \, \be=\de_{\mbf r}, \, \, {\mbf r }\notin G_B, \,  {\mbf r }\in \Z^N \\
 	\mf h \ot t^{\mbf r} , \, \text{if} \, \, \be=\de_{\mbf r}, \, \, {\mbf r }\in G_B.
 \end{cases} $$
Note that $\tau_B$ is a subquotient of full toroidal Lie algebra. So all the axioms of EALA can be easily verified for $\tau_B$ except  EA1. Define a bilinear form by

 \begin{center}
 	$(X({\mbf r}),Y({\mbf s}))=<X,Y>\de_{\mbf r+\mbf s,\mbf 0},$ for all $X,Y \in \mf g$, $\mbf r, \mbf s$ \ $\in \Z^N$,\\
 	$(h_{\mbf r},K(B{\mbf s}, {\mbf s}))=\de_{\mbf r+ \mbf s,\mbf 0}(B{\mbf r}|B{\mbf s})$ for all $\mbf r, \mbf s $ $\notin G_B$,\\
 	$(D{(\mbf u},\mbf 0),K{(\mbf v },\mbf 0))=(\mbf u| \mbf v)$ for all $\mbf u, \mbf v$ $\in \C^N$ and all other values are zero.
 	 \end{center} 
 Note that this form descends from the form of $\tau(S_N)$ after noting that $(D(B{\mbf r},{\mbf r})|K_B)=0$.	It is easy to verify that the form is invariant. Now to check that the form is non degenerate consider $\mbf r + \mbf s=\mbf 0$ and note that $(D({B\mbf r, \mbf r}),K(B{\mbf s}, {\mbf s}))=-(B{\mbf r},B{\mbf r})\neq 0$ for all ${\mbf r} \notin G_B$ and $(X({\mbf r}),Y({\mbf s}))=<X,Y>\neq 0$, for $X \in \mf g_\al$ and $Y \in \mf g_{-\al}$. (8) is standard to prove by definition of the core of EALA with the help of the root system as described above.  This completes the proof.

\end{proof}
\begin{remark}\label{r3.1}
	We call $\wtil{ H_B}$, which is naturally a subalgebra of $Der \mcal A$ (see Section 2), the derivation algebra of the SSEALA associated with $B$.
	
\end{remark}

\begin{example} \label{e3.1}
	
Consider the skew symmetric matrix $B=J_m$ as in Example 2.3. Then
the Lie algebra $\tau_{J_m}$ is a SSEALA which is the HEALA considered in \cite{RH}. Note that in this example the skew symmetric bilinear form corresponding to $J_m$ is non-degenerate and rank of $G_B$ is 0.

\end{example}
\begin{example}\label{e3.2}
 Consider the matrix 
	
	$$\bar J = \begin{pmatrix}
		J_m  &  & & 1 \\
		   &  & & \vdots \\
		   &  & &  1 \\
		-1 & \hdots &-1 & 0\\   
	\end{pmatrix} $$
The Lie algebra $\tau_{\bar J}$ is another example of SSEALA. Note that the skew symmetric bilinear form corresponding to $\bar J$ is degenerate and rank of $G_{\bar J}$ is 1. This SSEALA was constructed in \cite{RH} and named as KEALA.	
	 
\end{example}
 
\begin{lemma}\label{l3.1}
	Let $V$ be a finite dimensional vector space of dimension $N$. Let $\phi$ be a skew symmetric bilinear form on $V$. Let $B $ be the skew symmetric matrix corresponding to  $\phi$ with respect to some basis of $V$. Then\\
	(1) The rank of $B$ is even, say $2m$.\\
	(2) Let $k$ be the dimension of the radical. Then there exists a basis $v_1,v_2, \dots v_{N}$ of $V$ such that the corresponding matrix for  $\phi$ is 	$$B'=  \begin{pmatrix}
		J_{m} &0_{2m \times k} \\
		0_{k \times 2m} & 0_{k \times k} \\ 
	\end{pmatrix} .$$ 
In particular, $N=2m+k$.\\
(3) There exists $A \in GL(N, \C)$ such that $ABA^T=B'$. 

\end{lemma}
\begin{proof}
	For the proof of (1) and (2) see \cite{Sam P} Theorem 5-11, also see Lemma 1.1.5 of \cite{GW} for a refined proof for case $k=0$. Now from (1) and (2), it follows that there exists a $A \in GL(N, \C)$ such that $ABA^{-1}=B'$. It is known that $A$ can be chosen to be orthogonal, i.e $A^T=A^{-1}$ (One can see the Theorem 2.1 of \cite{DKK}, also see \cite{FRG} for the original proof). This completes the proof of (3). 
\end{proof}


\begin{proposition}\label{p3.2}
	Let $B$ be a skew symmetric matrix in $\mathfrak{gl}(N,\C)$ and $A \in GL(N,\Z)$ such that $B=A^TB'A$ for some $B'$. Then $\tau_B  \simeq \tau_{B'}$.
\end{proposition}
\begin{proof}
	Consider the automorphism of full toroidal Lie algebra $\til \tau$ given by $\Phi_A$ (\ref{a2.5}) corresponding to the matrix $A$. Note that $\Phi_A(D(B{\mbf r},{\mbf r}))=D((A^T)^{-1})B{\mbf r},A{\mbf r})=D(B'A{\mbf r},A{\mbf r}). $ Since $A$ is an invertible linear operator from the above we have $\Phi_A(H_B)=H_{B'}.$ Also observe that $\Phi_A(K({\mbf u},{\mbf r}))=(K(A{\mbf u},A{\mbf r}))  $ and $(B'A{\mbf u}|A{\mbf r})= ((A^T)^{-1}B{\mbf u}|A{\mbf r})= (A^T(A^T)^{-1}B{\mbf u}|{\mbf r})= (B{\mbf u}|{\mbf r}).$ Hence $\Phi_A$ maps $\mcal K_B$ into $\mcal K_{B'}$. Therefore $\Phi_A$ induces an isomorphism between $\tau_B$ and $\tau_{B'}$.
\end{proof}
Suppose $B=A^TB'A$ for some $A,B,B'$ as in Lemma \ref{l3.1}. Note that $A^TB'A$ is obtained from $B'$ by simultaneous row and column operations. Now in the case when $B'= \begin{pmatrix}
		J_m &0_{2m,1} \\
		0_{1, 2m} & 0 \\ 
	\end{pmatrix} = J' (say)$ and $B=\bar J$ the simultaneous row and column operations induces only integer matrices with entries in the set $\{1,0,-1\}$. Now instead of providing the elementary row and column operations method we just provide the matrix $A$, which is given by 
	$$ \begin{pmatrix}
		1 & 0 & 0 \hdots & 0  & 0 \hdots & 0 &0\\
		0 & I_{m-1} & &   \vdots &  -I_{m-1}  &         &  0   \\
		\vdots &  & &   & &       &    \vdots      \\
		0 & \hdots &0 & 0 & 0&\hdots 0 &     1     \\  
	   -1 &   0_{m-1}       &  & 0  & I_{m-1} &         &    0       \\
   \vdots &          &  & \vdots  &   &       & \vdots                   \\ 
         
		-1 & -1 &\hdots & 1 & \hdots & 1 & -1   &                       \\

	\end{pmatrix}_. $$
Now it is easy to check that $A\bar JA^T=J'$ and $A \in GL(2m+1,\Z)$. Therefore with the help of Proposition \ref{p3.2} we have  $\tau_{J'} \simeq \tau_{\bar J}$.

\section{Classification of irreducible modules for $\wtil{H_B} \ltimes \mcal A$.}
For the rest of the paper we consider $N=2m$ (except for section 6 where we assume $N=2m+1$)  and $B$ as a non-degenerate skew symmetric matrix. Consider the action of $\wtil{H_B}$ on $\mcal A$ as
\begin{align}\label{a4.1}
	D(B{\mbf r},\mbf r).t^{\mbf s}=(B{\mbf r| \mbf  s})t^{\mbf r+s}\\
	D({\mbf u},\mbf 0).t^{\mbf s}=({\mbf u|\mbf  s})t^{\mbf s},
\end{align} 
for all ${\mbf r,\mbf s} \in \Z^N$ and $\mbf u$ $ \in \C^N$.
This action defines a Lie algebra structure on $\wtil{H_B} \ltimes \mcal A$. In this section we classify irreducible modules for $\wtil{H_B} \ltimes \mcal A$ with finite dimensional weight spaces and use this classification to classify modules for SSEALA $\tau_B$.\\
The following concept can be found in \cite{GSR}. Suppose $\mcal V$ is a Lie algebra acting on a commutative associative unital algebra $\mcal B$ by derivation.  

\begin{definition}\label{d4.1}
	A $\mcal{BV}$ module $M$ is a vector space with actions of $\mcal B$ and $\mcal V$ which are compatible via the Leibniz rule 
	$$ \eta(fm)=\eta(f)m+ f (\eta m),  $$
	$$ fg(m)=f(g(m))  , \,\, 1.m=m ,$$
	for all $\eta \in \mcal V, \, f ,g\in \mcal B, \, m \in M$ and $1$ unit in $\mcal B$.\\
Equivalently, a $\mcal {BV}$ module structure may be expressed via smash product $\mcal B \# U(\mcal V)$, where the universal enveloping algebra $U(\mcal V)$ viewed as Hopf algebra with co-product $\Delta(u)=\sum u_1 \ot u_2$. The smash product is the associative algebra structure on a vector space $\mcal B \ot U(\mcal V)$ with the product
$$(f \ot u) (g \ot v) =\sum{fu_1(g) \ot u_2(v)},  $$ for all $f,g \in \mcal B, \, u,v \in U(\mcal V)$. The $\mcal {BV}$ module structure is equivalent to the action of associative algebra $\mcal B \ot U(\mcal V)$. 
\end{definition}
\begin{remark}
In accordance with the current literature in this paper a $\mcal{BV}$ module is called a jet module for $\mcal V$, when $\mcal B= \mcal A$. 
	
\end{remark}


Let $V$ be an irreducible jet module for $\wtil{H_B} $ with finite dimensional weight spaces with respect to $\mcal D \oplus \C$.
 Choose a weight $\lambda \in (\mcal D \oplus \C)^*$ such that $V_\lambda \neq 0$. Then due to irreducibility of $V$ we have $V=\displaystyle{\bigoplus_{\mbf m \in \mathbb{Z}^N}}V_{\mbf m},$ where $V_{\mbf m}= \{ v \in V: D(\mbf u,\mbf 0)v=(\mbf u|\mbf m+\bs \alpha)v $ for all $\mbf u \in \mathbb{C}^N \}$ and $\bs \alpha=(\lambda(d_1),\lambda(d_2),....,\lambda(d_N)\in \mathbb{C}^N.$ Let $U$ denote the universal enveloping algebra of $\wtil {H_B}\ltimes \mcal A.$ Let $L$ be the two sided ideal of $U$ generated by $t^{\mbf r}t^{\mbf s}-t^{\mbf r+\mbf s}$ and $t^{\mbf 0}-1$. Consider $U'=U/L$ an associative algebra. Note that $U'$ is a $\Z^N$ graded algebra, in fact $U'= \dis{\bigoplus_{{\mbf r}\in \Z^N}}U'_{\mbf r}, $ where $U'_{{\mbf r}}= \{X \in U':[D({\mbf u}, \mbf 0),X]=({\mbf u| \mbf r}) X, \forall \, {\mbf u} \in \C^N \}$. Obviously $U'_\mbf 0$ is an associative algebra.\\

Let $T({\mbf r})=t^{\mbf {-r}}D({B\mbf r, \mbf r})-D(B{\mbf r},\mbf 0)$, for all $\mbf r \in \Z^N$ and $\mcal T_B=span \{ T({\mbf r}): {\mbf r} \in \Z^N\}$. Then it is easy to check that $[T({\mbf r}),T({\mbf s})]=(B{\mbf r}|{\mbf s})[T({\mbf r +\mbf s})-T({\mbf r })-T({\bf s})].$ In particular, $\mcal T_B$ is a Lie subalgebra of $U'$
with respect to the commutator bracket.
\begin{theorem}\label{t4.1}
	
	(1) Let $V$ be an irreducible jet module for $\wtil{H_B} $ with finite dimensional weight spaces. Then $V \simeq V_\mbf 0 \otimes \mcal A$, where $V_\mbf 0$ is an irreducible module for $\mcal T_B$.\\
	(2) Let $V_\mbf 0$ be a finite dimensional irreducible module for $\mcal T_B$. Then $V_\mbf 0$ is an irreducible module for $\mf{sp}_{2m} \op R$ and $R$ is central. Conversely any finite dimensional irreducible module for $\mf {sp}_{2m} \op R$ can be extended to an irreducible module for $\mcal T_B$.

\end{theorem}
The aim of rest of this section is to prove Theorem \ref{t4.1}. We start with the following Lemma.

\begin{lemma}\label{l4.1}
	$U(\mcal T_B)=U'_\mbf 0$.
\end{lemma}
\begin{proof}
 Note that $U(\mcal T_B) \subseteq U_\mbf 0'$ is obvious.	Let $X \in U_\mbf 0'$. Then by PBW theorem we have $X$ is a linear combination of monomials $t^{-\mbf r_1}\dots t^{-{\mbf r_l}}D(B{\mbf s_1}, \mbf s_1) \dots D(B{\mbf s_k}, \mbf s_k)$, where ${\mbf r_i, \mbf s_i} \in \Z^N$, $l,k \in \N$ and $\dis{\sum_{i}}{\mbf r_i}= \dis{\sum_{i} }{\mbf s_i}$. Now using the fact $[D(B{\mbf r}, \mbf r),t^{\mbf s}]=(B{\mbf r}|{\mbf s})t^{\mbf r +s}$ and $t^{\mbf r}t^{\mbf s}=t^{\mbf r+ \mbf s},$ we see that each monomial can be written as sum of monomials of the form $t^{-\mbf r_1}\dots t^{-{\mbf r_l}}D(B{\mbf r_1}, \mbf r_1) \dots D(B{\mbf r_l}, \mbf r_l)$. Hence $X \in U(\mcal T_B)$.
		
	\end{proof}
\begin{lemma}\label{l4.2}

		(1) Each $V_{\mbf r}$ is an irreducible representation of  $\mcal T_B$. \\
		(2) $V_{\mbf r} \simeq V_{\mbf s}$ as a $\mcal T_B$-module.\\
		(3) $t^{\mbf s}.V_{\mbf r}=V_{\mbf r + \mbf s}$.
	
\end{lemma}
\begin{proof}
	Let $u,v \in V_{\mbf r}$. Then by irreducibility of $V$ and weight arguments we find $X \in U_\mbf 0'$ such that $X.u=v$. Hence $(1)$ follows by Lemma \ref{l4.1}. Note that $t^{\mbf s}t^{-{\mbf s}}=1$ implies $t^{\mbf s}$ is invertible, this proves $(3)$. To prove $(2)$ define a map $f:V_{\mbf r} \to V_{\mbf s}  $ given by $ f(v)=t^{\mbf s-\mbf r}v$. Now \\
	$$f(T({\mbf k})v)=t^{\mbf s- \mbf r}T({\mbf k})v=T({\mbf k})t^{\mbf s- \mbf r}v=T({\mbf k})f(v),$$
	since $t^{\mbf r}$ commutes with $T({\mbf k})$. Hence $f$ is a $\mcal T_B$ module map. It is easy to see that $f$ is both injective and surjective.

\end{proof}
Now it is clear that, $V \simeq V_\mbf 0 \otimes \mcal A$, where $V_\mbf 0$ is a finite dimensional irreducible $\mcal T_B$-module. Conversely given any finite dimensional irreducible $\mcal T_B$-module $V_\mbf 0$ we can define a $\wtil {H_B} \ltimes \mcal A$-module structure on $V_\mbf 0 \otimes \mcal A$. Let ${\boldsymbol{\be}} \in \C^N$ and define \begin{align}\label{a4.3}
	 D(B{\mbf r},{\mbf r}).v \ot t^{\mbf k}=(B{\mbf r}|{\mbf k + \bs{\be}} )v\ot t^{\mbf k+ \mbf r}+(T({\mbf r})v)\ot t^{\mbf k+ \mbf r}. \\
	 t^{\mbf r}.v \ot t^{\mbf k}=v\ot t^{\mbf k+\mbf r}.
\end{align}

Now we check that (4.3) and (4.4) defines an $\wtil {H_B} \ltimes \mcal A$-module structure on $V_\mbf 0 \ot \mcal A$. It is easy to check that $[D(B{\mbf r},{\mbf r}),t^{\mbf s}](v\ot t^{\mbf k})=(B{\mbf r}|{\mbf s})t^{\mbf r+ \mbf s}(v\ot t^{\mbf k}).$ We check that  $[D(B{\mbf r},{\mbf r}),D(B{\mbf s},{\mbf s})](v\ot t^{\mbf k})=(B{\mbf r}|{\mbf s})D(B({\mbf r + \mbf s}), {\mbf r+ \mbf s})(v \ot t^{\bf k}).$ \\

Consider $ [D(B{\mbf r},{\mbf r})D(B{\mbf s},{\mbf s}) -D(B{\mbf s},{\mbf s}) D(B{\mbf r},{\mbf r})].v\ot t^{\bf k}$\\

$=(B{\mbf s}|{\mbf k+\bs \be})\{ (B{\mbf r}|{\mbf k+\mbf s+\bs \be})v  +T({\mbf r})v \}\ot t^{\mbf k+\mbf  r+ \mbf s} +\{ (B{\mbf r}|{\mbf k+\mbf s+\bs \be})T({\mbf s})v+T({\mbf r})T({\mbf s})v \}\ot t^{\mbf k+\mbf r+ \mbf s} $ $-(B{\mbf r}|{\mbf k+\bs \be})\{ (B{\mbf s}|{\mbf k+ \mbf r+\bs \be})v  +T({\mbf s})v \}\ot t^{\mbf k+\mbf r+\mbf s} -\{ (B{\mbf s}|{\mbf k+ \mbf r+\bs \be})T({\mbf r})v+T({\mbf s})T({\mbf r})v \}\ot t^{\mbf k+ \mbf r+ \mbf s} $ \\

$=\{(B{\mbf s}|{\mbf k + \bs \be}) (B{\mbf r}|{\mbf k+\mbf s+\bs \be})-(B{\mbf r}|{\mbf k + \bs \be})(B{\mbf s}|{\mbf k+\mbf r +\bs \be})\}v\ot t^{\mbf k+\mbf s+\mbf r} +(B{\mbf r}|{\mbf s})T({\mbf s})v \ot t^{\mbf k+\mbf s+\mbf r} -(B{\mbf s}|{\mbf r})T({\mbf r})v \ot t^{\mbf k+\mbf s+\mbf r} +[T({\mbf r}),T({\mbf s})]v \ot t^{\bf k+r+s}$\\

$=\{(B{\mbf s}|{\mbf k + \bs \be}) (B{\mbf r}|{\mbf s})-(B{\mbf r}|{\mbf k + \bs \be})(B{\mbf s}|{\mbf r })\}v\ot t^{\mbf k+\mbf s+\mbf r} +(B{\mbf r}|{\mbf s})\{T({\mbf s})+T({\mbf r})\}v \ot t^{\mbf k+\mbf s+\mbf r} +[T({\mbf r}),T({\mbf s})]v \ot t^{\bf k+\mbf r+\mbf  s}$\\

$=(B{\mbf r}|{\mbf s})(B({\mbf r+\mbf s})|{\mbf k +\bs \be})v \ot t^{\mbf k+\mbf r+\mbf s} +(B{\mbf r}|{\mbf s})T({\mbf r+\mbf  s})v \ot t^{\mbf k+\mbf r+\mbf s} $\\

$=(B{\mbf r}|{\mbf s})D(B({\mbf r + \mbf s}), {\mbf r+ \mbf s})v \ot t^{\mbf k}.$\\

Therefore to classify irreducible modules for $\wtil 
{H_B} \ltimes \mcal A$ with finite dimensional weight spaces, it is sufficient to classify finite dimensional irreducible $\mcal T_B$-modules. For this we proceed similarly to \cite{R04}. For $q \in \N$, define 
 $$T_q({\mbf s, \mbf r_1,\dots, \mbf r_q})= T({\mbf s})+\dis{\sum_{1 \leq k \leq q} \, \, \sum_{1 \leq i_1 < \cdots <i_k \leq q}}(-1)^k (  T({\mbf s+\mbf r_{i_1}+\dots+\mbf r_{i_k}}).$$ In the following lemma we collect properties of $T_q({\mbf s, \mbf r_1,\dots, \mbf r_q})$.

\begin{lemma}\label{l4.4}
	
		(1) $T_q({\mbf s, \mbf r_1,\dots, \mbf r_q})=T_q({\mbf s, \mbf r_{\sigma(1)},\dots,  \mbf r_{\sigma(q)}}), $ for any permutations $\sigma$ on $q$ letters.\\
		(2) $T_q({\mbf s, \mbf r_1,\dots, \mbf r_q})= T_{q-1}({\mbf s, \mbf r_1,\dots,\hat{\mbf r}_j ,\dots, \mbf r_q})-T_{q-1}({\mbf s + \mbf r_j, \mbf r_1,\dots, \hat{\mbf r}_j, \dots, \mbf r_q})$, for all $1 \leq j \leq q$,  $\widehat .$ means omission.  \\
		(3) $T_q({\mbf s, \mbf r_1,\dots, \mbf r_q})=0$ if $\mbf r_j=\mbf 0$ for some $1\leq j \leq q$.\\
		($4$) $[T({\mbf r}),T_q({\mbf s, \mbf r_1,\dots, \mbf r_q}) ]
		= -(B{\mbf r}|{\mbf s})T_{q+1}({\mbf s, \mbf r, \mbf r_1,\dots, \mbf r_q}) +$
		$$\dis{\sum_{i=1}^{q}}(B{\mbf r}|{\mbf r_i})T_q({\mbf s +\mbf r_i, \mbf r_1,\dots,\hat{\mbf r}_i, \dots,\mbf  r_q, \mbf r})$$ for all $q \geq 2$.\\

\end{lemma}  
\begin{proof}
	Note that $(1)$ follows from the definition. To prove $(2)$ collect all the terms where ${\mbf r_j}$ does not appear in the expressions of $T_q({\mbf s, \mbf r_1,\dots, \mbf r_q})$. It is easy to see that these terms equal to $T_{q-1}({\mbf s, \mbf r_1,\dots,\hat{\mbf r}_j ,\dots, \mbf r_q})$. On the other hand sum of the remaining terms of $T_q({\mbf s, \mbf r_1,\dots, \mbf r_q})$ is $-T_{q-1}({\mbf s +\mbf r_j, \mbf r_1,\dots, \hat{\mbf r}_j, \dots, \mbf r_q})$. This proves $(2)$. $(3)$ immediately follows from $(2)$. We prove $(4)$ by induction principal. Note that,
  \begin{center}
  	$[T({\mbf r}),T_2({\mbf s, \mbf r_1, \mbf r_2})]=[T({\mbf r}),T({\mbf s})-T({\mbf s+\mbf r_1})-T({\mbf s+ \mbf r_2})+T({\mbf s+ \mbf r_1+ \mbf r_2})]$\\
  	$=(B{\mbf r}|{\mbf s})\{T({\mbf r + \mbf s})-T({\mbf r})-T({\mbf s}) \}$$- \dis{\sum_{i=1}^{2}}(B{\mbf r}|{\mbf r_i+\mbf s})\{T({\mbf r +\mbf s+\mbf r_i})-T({\mbf r})-T({\mbf s+\mbf r_i}) \} +(B{\mbf r}|{\mbf s+\mbf r_1+\mbf r_2})\{ T({\mbf r +\mbf s+\mbf r_1+ \mbf r_2})-T({\mbf r})-T({\mbf s+ \mbf r_1+ \mbf r_2})\} $ \end{center}
  	$=-(B{\mbf r}|{\mbf s})T_{3}({\mbf s, \mbf r, \mbf r_1, \mbf r_2}) +\dis{\sum_{i=1}^{2}}(B{\mbf r}|{\mbf r_i})T_2({\mbf s +\mbf r_i, \mbf r_1,\dots,\hat{\mbf r}_i, \dots, \mbf r_2, \mbf r})$.\\ Hence the claim is true for $q=2$. 
  	Now assume that the claim is true for $q$. Then \\
 $[T({\bf r}),T_{q+1}({\mbf s, \mbf r_1,\dots, \mbf r_q}) ]$\\
 $=[T({\mbf r}),T_{q}({\mbf s, \mbf r_1 ,\dots, \mbf r_q})-T_{q}({\mbf s+\mbf r_{q+1}, \mbf r_1,\dots, \mbf r_{q}})] $\\
$= -(B{\mbf r}|{\mbf s})T_{q+1}({\mbf s, \mbf r, \mbf r_1,\dots, \mbf r_q}) +
\dis{\sum_{i=1}^{q}}(B{\mbf r}|{\mbf r_i})T_q({\mbf s +\mbf r_i, \mbf r_1,\dots,\hat{\mbf r}_i, \dots, \mbf r_q, \mbf r})$\\
$ +(B{\mbf r}|{\mbf s+ \mbf r_{q+1}})T_{q+1}({\mbf s+ \mbf r_{q+1}, \mbf r, \mbf r_1,\dots, \mbf r_q}) -
\dis{\sum_{i=1}^{q}}(B{\mbf r}|{\mbf r_i})T_q({\mbf s+ \mbf r_{q+1} +\mbf r_i, \mbf r_1,\dots,\hat{\mbf r}_i, \dots, \mbf r_q, \mbf r})$\\
$=  -(B{\mbf r}|{\mbf s})T_{q+2}({\mbf s, \mbf r, \mbf r_1,\dots, \mbf r_{q+1}}) +\dis{\sum_{i=1}^{q+1}}(B{\mbf r}|{\mbf r_i})T_{q+1}({\mbf s + \mbf r_i, \mbf r_1,\dots,\hat{\mbf r}_i, \dots, \mbf r_{q+1}, \mbf r})  .$ 
Therefore by induction principal the claim is true for all $q \geq 2$. 
\end{proof}
For each $q \in \N$, define $I_q= span \{T_q({\mbf s, \mbf r_1,\dots,\mbf r_q}): {\mbf s, \mbf r_1,\dots, \mbf r_q} \in \Z^N\}$. It is clear that $I_1=\mcal T_B$. Now we record some properties of $I_q$.
\begin{lemma}\label{L 4.4}
	($1$) $I_{q+1} \subseteq I_{q}$, for all $q \in \N$.\\
	(2) $I_q$ is an ideal of $\mcal T_B$ for all $q \geq 1$.\\
	(3) $[I_p,I_q]\subseteq I_{p+q-2}$ for all $p+q\geq 3$ and 
	$[I_1,I_1]\subseteq I_1$.
\end{lemma}
\begin{proof}
	(1) follows from Lemma 4.3 (2). (2) follows from Lemma 4.3(4) and (1). We prove (3) in appendix.
\end{proof}
\begin{lemma}\label{l4.5}
	
		(1) $T_q({\mbf s, \mbf r_1,\dots,\mbf r_q}) \notin I_{q+1}$ .\\
		(2) $\dis{\bigcap_{q\geq 1}}I_q=\{0\}$.\\
		(3) $T_q({\mbf s, \mbf r_1,\dots,\mbf r_q})+T_q({\mbf s, \mbf n_1,\dots,\mbf r_q})=T_q({\mbf s, \mbf r_1 +\mbf n_1,\dots,\mbf r_q}) $ mod $I_{q+1}$.\\
		($4$) $T_q({\mbf s, -\mbf r_1,\dots,\mbf r_q})=-T_q({\mbf s -\mbf r_1, \mbf r_1,\dots,\mbf r_q})$.
	
\end{lemma}
\begin{proof}
	Define a linear map $\eta:\mcal T_B \to \mcal A$ such that $\eta(T(\mbf s))=t^{\mbf s}$, for all $\mbf s \in \Z^N$. Set $P_q({ \mbf s,\mbf r_1,\dots,\mbf r_q})=t^{\mbf s}\prod_{i=1}^{q}(1-t^{\mbf r_i}) $, for all $\mbf s,\mbf r_1,\dots,\mbf r_q \in \Z^N$.   \\
	{\bf Claim:} $\eta(T_q({\mbf s, \mbf r_1,\dots, \mbf  r_q}))=P_q({\mbf s,\mbf  r_1,\dots, \mbf  r_q})$ for all $q \geq 1$. We prove it by induction on $q$. First note that \\
	$\eta(T_1({\mbf s, \mbf r_1}))= \eta(T(\mbf s)) -\eta(T(\mbf s +\mbf r_1))=t^{\mbf s}(1-t^{\mbf r_1})=P_1({\mbf s,\mbf r_1})$. Thus the claim is true for $q=1$. Let us assume that claim is true for $q=k-1$, for some $k >1$. Now by Lemma \ref{l4.4}(2) we have,\\
	\begin{align*}
	\eta(T_k({\mbf s, \mbf r_1,\dots, \mbf r_k})) & = \eta(T_{k-1}({\mbf s, \mbf r_2,\dots, \mbf r_k}))- \eta(T_{k-1}({\mbf s +\mbf r_1, \mbf r_2,\dots, \mbf r_k})) \cr
	&= P_{k-1}({\mbf s, \mbf r_2,\dots, \mbf r_k})-P_{k-1}({\mbf s+\mbf r_1, \mbf r_2,\dots, \mbf r_k}) \cr
	&=t^{\mbf s}\prod_{i=2}^{k}(1-t^{\mbf r_i})-t^{\mbf s+ \mbf r_1}\prod_{i=2}^{k}(1-t^{\mbf r_i})\cr
	&= P_{k}({\mbf s, \mbf r_2,\dots, \mbf r_k})
	\end{align*}
This proves the claim by principal of induction. For all $q \geq 1$ define, $S_q=span \{P_q({ \mbf s,\mbf r_1,\dots,\mbf r_q}): \mbf s,\mbf r_1,\dots,\mbf r_q \in \Z^N\}$. Clearly $S_q$ is an ideal of $\mcal A$. Moreover, from the above claim we get $\eta(I_q)=S_q$, for all $ q \geq 1$.\\
 Now to prove (1), it is sufficient to prove that $P_q({ \mbf s,\mbf r_1,\dots,\mbf r_q}) \notin S_{q+1}$, which will follow with the same method as Claim 1 and Claim 2 of Lemma 3.5 (\cite{R04}).

	To prove $(2)$ first observe that $S_q =S_1^q$, for all $q \geq 1$. Since $\eta$ maps basis of $\mcal T_B$ to linearly independent set, so $\eta$ is injective. Therefore to prove $(2)$, it is sufficient to prove that $\dis{\cap_{q \geq 1}}S_q=0$. Let $f \in \dis{\cap_{q \geq 1}}S_q$. Without loss of generality, we can assume that
	$f=\dis{\sum_{i}^{a}c_{\mbf m_i}t^{\mbf m_i}}$, where ${\bf m_i}=(m^1_i,\dots m^N_i)$ and $m_i^j \geq 0$ for all $1 \leq i \leq a$ and $ 1\leq j \leq N$. If not, we multiply $f$ by a monomial $g$ such that $gf=  \dis{\sum_{i}^{a}c_{\mbf m_i}t^{\mbf n_i}} \in \dis{\cap_{q \geq 1}}S_q $ and each components of $\mbf n_i$, for $1 \leq i \leq a$ becomes non-negative, where $t^{\mbf n_i}=gt^{\mbf m_i}$.\\
	 Let $l_j=max\{m_i^j: 1 \leq i \leq a \} $, for $ 1 \leq j \leq N$. Now consider the operator $d=(\frac{d}{dt_1})^{l_1} \dots (\frac{d}{dt_1})^{l_N}$. Then we have $df \neq 0$. Let $\dis{\sum_{i}^{N}l_i}=L.$ Now $f \in S_{L+1}$, hence every non zero terms of $df$ has a factor of the form $(1-t^{\bf m})$, for some $\mbf m \in \Z^N$. Therefore $df(1,1,..,1)=0$, a contradiction. Thus $\dis{\cap_{q \geq 1}}S_q=\{0\}$.\\
	 The proof of $(3)$ follows on similar lines as $(2)$ of Lemma 3.5, \cite{R04}. To prove (4) one can just note that $\eta(T_q({\mbf s, -\mbf r_1,\dots,\mbf r_q})+T_q({\mbf s -\mbf r_1, \mbf r_1,\dots,\mbf r_q}))=0$, now the proof follows from the fact that $\eta $ is injective.


\end{proof}
\begin{lemma}\label{l4.6}
	
		(1) $T_q({\mbf s, \mbf r_1,\dots, \mbf r_q})=T_q({\mbf r, \mbf r_1,\dots, \mbf r_q})$ mod $I_{q+1}$, for all $\mbf s, \mbf r ,\mbf r_1,\dots, \mbf r_q \in \Z^N$.\\
		(2) $-T_q({\mbf 0, \bf r_1,\dots, r_q})= T_q({\mbf 0, \bf -r_1,\dots, r_q})$ mod $I_{q+1}$.\\
		(3) dim $I_q/I_{q+1} \leq N^q$.\\
		($4$) $I_q$ is a co-finite ideal in $\mcal T_B$.
	
\end{lemma}
\begin{proof}
(1) follows from Lemma \ref{l4.4}(2), (2) follows from Lemma \ref{l4.5}(3) and proof of (3),(4) is similar to Lemma 3.6 of \cite{R04}.
\end{proof}
\begin{lemma}\label{l4.7}
Let $J$ be a co-finite ideal in $\mcal T_B$. Then $J$ contains $I_k$ for sufficiently large $k.$ 
\end{lemma}
\begin{proof}
	Consider the canonical projection $\pi: \mcal T_B \to \mcal T_B/J$. Note that the image of $I_k$ under $\pi$ is isomorphic to $I_k/({J\cap I_k})$. Now from Lemma 4.4(1) and Lemma 4.5(1) we have $I_{k+1} \subsetneq I_k$ for all $k \geq 1$. Therefore images of $I_k$ under $\pi$ forms a decreasing sequences of ideals of $\mcal T_B/J$, for all $k \geq 1$. Again by hypothesis $J$ is a cofinite ideal, thus $\{I_k/(J \cap I_k)\}_{k \geq 1}$ is a decreasing sequence of finite dimensional algebras. Hence there exists a $k \in \N$ such that dim $I_k/(J\cap I_k) =$ dim $I_{q}/(J\cap I_{q})$, for some $k \geq 1$,  for all $q \geq k$. Thus it is sufficient to prove that $I_k/(J\cap I_k)=0$.\\ 
	Indeed, let $0 \neq x_k \in I_k/(J\cap I_k)$. Since dim $I_k/(J\cap I_k) =$ dim $I_{k+q}/(J\cap I_{k+q})$ and $\{I_k/J\cap I_k\}$ is a family of decreasing sequences, there exists $x_{k+q} \in I_{k+q}/(J\cap I_{k+q})$ such that $x_k=x_{k+q} $ mod $I_k \cap J$, i.e $x_k-x_{k+q}=0$ in  $I_k/I_k \cap J.$ This means that $x_k-x_{k+q}=0$ in  $I_{k+1}/I_{k+1} \cap J.$ Hence there exists $y_{k+q} \in I_{k+q}\cap J$ such that $x_k-x_{k+q}=y_{k+q}$, i.e $x_k \in I_{k+q}$ for all $q \geq 1$. This proves that $x_k \in \cap_{k \geq 1} I_k$, a contradiction to Lemma \ref{l4.5}(2).
	
\end{proof}
Now we record a Lemma from \cite{JH} which we use to prove Proposition 4.1.
\begin{lemma}( Proposition 19.1,\cite{JH})\label{l4.3}
	Let $\phi: \mf L \to \mf{gl}(V)$ be a finite dimensional irreducible representation of a Lie algebra $\mf L$. Then $\phi({\mf L})$ is a reductive Lie algebra with at most one dimensional center. 
\end{lemma}
\begin{proposition}\label{p4.1}
	Let $V_0$ be a finite dimensional irreducible module for $\mcal T_B$. Then $I_3$ acts trivially on $V_0$.
\end{proposition}
\begin{proof}
	Let $\phi: \mcal T_B \to \mf{gl}(V_0)$ be the representation homomorphism. Then by Lemma \ref{l4.7} $I_{k+1} \subseteq ker \phi$ for sufficiently large $k$. Now by Lemma \ref{L 4.4}(4) $[I_k,I_k] \subseteq I_{2k}$ and $I_{2k} \subseteq ker \phi$. Therefore $\phi(I_k)$ is an abelian ideal. Note that Lemma \ref{l4.3} imples that $\phi(I_k)$ is central and hence $\phi[\mcal T_B,I_k]$ acts trivially on $V_0$. Now Lemma \ref{l4.4}(4) gives 
	$$[T({\mbf m}),T_k({\mbf r,\mbf s,\dots,\mbf s})] =-(B{\mbf m}|{\mbf r})T_{k+1}({\mbf r, \mbf m, \mbf s,\dots, \mbf s}) +
	k(B{\mbf m}|{\mbf s})T_k({\mbf r + \mbf s, \mbf s,\dots, \mbf s, \mbf m}),$$ for all $\mbf m, \mbf r, \mbf s \in \Z^N$. \\
	{\bf Claim:} $T_k({\bf a,m,s,\dots,s}) \in ker \phi$ for all ${\bf a, m , s} \in \Z^N$.\\
	Case (i): Let us fix $\mbf m, \mbf s$ such that $(B{\bf m}|{\bf s}) \neq 0$.
	 Since $I_{k+1}$ and $[\mcal T_B,I_k]$ contained in $ker \phi$, from the above formula we have $T_k({\bf r+s,s,\dots,s, m}) \in ker \phi$. Since $\mbf r \in \Z^N$ is arbitrary, we have  $T_k({\bf a,s,\dots,s,m}) \in ker \phi$ for arbitrary $ {\bf a} \in \Z^N$. Now the claim follows for this case from Lemma \ref{l4.4}(1). \\
	 Case (ii): Let us fix $\mbf m,\mbf s \in \Z^N$ such that $(B{\bf m}|{\bf s}) =0$. Clearly, if one of $\bf m,$ or $\bf  s$ is zero, then the claim follows immediately from definition. Let both $\bf m,s$ are non-zero. Let $U_{\mbf s}$ be the orthogonal complement of $B \mbf s$. Then $\Z^N \setminus (U_{\mbf s} \cap \Z^N)$ is a proper subset of $\Z^N$ as $B$ is non degenerate. Hence there exists $\mbf b \in \Z^N$ such that $(B{\mbf b}| \mbf s) \neq 0$. Therefore by case (i), we have $T_k({\mbf a,\mbf b, \mbf s,\dots,\mbf s}) \in ker \phi$ and $T_k({\bf a,m-b,s,\dots,\mbf s}) \in ker \phi$, as $(B({\bf m-b})|\mbf s) \neq 0$. Hence by Lemma \ref{l4.5}(3) we have 
	 $$T_k({\bf a,m,s,\dots,s})= T_k({\bf a,b,s,\dots,s})+T_k({\bf a,m-b,s,\dots,s}) + I_{k+1} \in ker \phi$$
	  for all ${\bf a} \in \Z^N$. This completes the proof of the Claim. \\ Now we consider the relation:
	  $$[T({\mbf m}),T_k({\mbf a, \mbf m_1, \mbf s\dots,\mbf s})] =-(B{\mbf m}|{\mbf a})T_{k+1}({\mbf a,\mbf m, \mbf m_1, \mbf s,\dots, \mbf s}) +
	  (B{\mbf m}|{\mbf m_1})T_k({\mbf a +\mbf m_1, \mbf s,\dots, \mbf s, \mbf m})    $$
	  $$+(k-1)(B{\mbf m}|{\mbf s})T_k({\mbf a + \mbf s,\mbf m_1, \mbf s,\dots,\mbf  s, \mbf m})  $$ for all $\mbf a, \mbf m_1, \mbf m, \mbf s \in \Z^N$.
	  Now continuing the above method we get $T_k({\mbf a,\mbf m_1,\mbf m_2, \mbf s\dots, \mbf s}) \in ker \phi$ for all $\mbf m_1, \mbf m_2, \mbf a, \mbf s \in \Z^N$. Continuing this process by we have $I_k \subseteq ker \phi$. Now instead of $I_{k+1}$ we consider $I_k \subseteq ker \phi$ and repeat the above process to deduce $I_3 \subseteq ker \phi$ (see Lemma \ref{L 4.4}(3) why it fails for $I_2$). 
\end{proof}
\begin{lemma}\label{l4.8}
 dim $\mcal T_B/I_3 =$ $4m^2+2m$.
\end{lemma}
\begin{proof}
	We know that $I_3 \subseteq I_2 \subseteq I_1=\mcal T_B$. It is clear that dim $\mcal T_B/I_3$= dim $\mcal T_B/I_2$ + dim $ I_2/I_3$. Note that by Lemma \ref{l4.5}(3) and \ref{l4.6}(2) we have $T_2(\mbf 0,\bf r, \bf s)$ is linear in both $\bf r, \bf s$ mod $I_3$. As $I_2/I_3$ is spanned by $T_2(\mbf 0,\mbf r, \mbf s)+I_3$, $\{ T_2(\mbf 0,{\mbf e_i},{\mbf e_j}) +I_3: 1 \leq i, j \leq 2m\}$ is a basis for $I_2/I_3$. By a similar argument we have $I_1$  is spanned by $T_1(\mbf 0, \bf r) $ mod $I_2$ and $T_1(\mbf 0, \bf r)$ is linear in $\bf r$ mod $I_2$. Hence $\{ T_1(\mbf 0,{\mbf e_i}): 1 \leq i \leq 2m\}$ forms a basis for $I_1/I_2$. This proves the lemma.
\end{proof}

 For a skew symmetric matrix $B$ define  $\mf{G}_B= \{X \in \mf{gl}(N,\C): BX=-X^TB \} $. It is easy to check that $\mf G_B$ is a Lie algebra and $ {\bf r}(B{\bf r})^T \in \mf G_B$.  Note that if $B$ can be written as $B=A^TB'A$ for some $A \in GL(N,\C)$ then $B'$ is also a skew symmetric matrix and $A \mf G_B A^{-1}=\mf G_{B'}$.\\
  It is well known that $\mf G_{J_m} \simeq \mf{sp}_{2m}$ and dimension of $\mf G_{J_m}$ is $2m^2+m$. Consider a linear map $\psi_B :\mcal T_B \to  \mf{G}_B$ given by
 $$ \psi_B(T({\bf r}))= {\bf r}(B{\bf r})^T,$$ here we treated ${\bf r}$ as column vector and hence $(B{\bf r})^T$ is a row vector for all ${\bf r } \in \Z^N$.
 \begin{lemma}
 	$\psi_B$ is a Lie algebra homomorphism.
 \end{lemma}
\begin{proof}
Note that, \begin{align*} 
	\psi_B([T({\bf r}), T(\mbf s)]) & = (B\mbf r|\mbf s )\psi_B(T(\mbf r+ \mbf s)-T(\mbf r)-T(\mbf s)) \cr
	 &= (B\mbf r|\mbf s ) \{ (\mbf r + \mbf s )(B({\mbf r + \mbf s}))^T - {\mbf r}(B{\mbf r})^T -{\mbf s}(B{\mbf s})^T  \} \cr
	 &= (B\mbf r|\mbf s )\{{\mbf r}(B{\mbf s})^T + {\mbf s}(B{\mbf r})^T\}
	\end{align*}
Again, \begin{align*}
	[{\mbf r}(B{\mbf r})^T, {\mbf s}(B{\mbf s})^T] &= {\mbf r}(B{\mbf r})^T{\mbf s}(B{\mbf s})^T-{\mbf s}(B{\mbf s})^T{\mbf r}(B{\mbf r})^T\cr 
	&= (B\mbf r|\mbf s ){\mbf r}(B{\mbf s})^T-(B\mbf s|\mbf r ){\mbf s}(B{\mbf r})^T \cr
	&= (B\mbf r|\mbf s )\{{\mbf r}(B{\mbf s})^T + {\mbf s}(B{\mbf r})^T\}.
\end{align*}
Thus we have $\psi_B([T({\bf r}), T(\mbf s)])=[\psi_B(T({\bf r})), \psi_B(T(\mbf s))])$, for all $\mbf r, \mbf s \in \Z^N$. This completes the proof.	
\end{proof}

	 \begin{lemma}\label{l4.9}
 $\psi_{J_m}:\mcal {T}_{{J_m}} \to \mf{sp}_{2m}$ is a surjective Lie algebra homomorphism.
 \end{lemma}

\begin{proof}
First note that ${\bf r}({J_m\mbf r})^T \in \mf{sp}_{2m}.$ For $1 \leq i\leq m$, take ${\mbf r= \mbf e_i}$, then ${\mbf r}({J_m\mbf r})^T=-{\mbf e_i \mbf e_{m+i}^T}=-E_{i,m+i}$. Again  for ${\mbf r =\mbf e_{m+i}}$, ${\mbf r}({J_m\mbf r})^T=-E_{m+i,i}$. It is easy to see that ${\mbf r}({J_m\mbf s})^T+{\mbf s}({J_m\mbf r})^T \in \mf{sp}_{2m}$. It follows from the following table that the image of $\psi_{J_m}$
contains the standard basis of $\mf{sp}_{2m}$ (see e.g. \cite{JH}, Section 1.2).
 
\begin{center}

 \begin{tabular} {|c|l|r|}
	
	\hline
	${\bf r}$ & ${\bf s}$ & ${\bf r}({J_m\mbf s})^T+{\bf s}({J_m\mbf r})^T$ \\
	\hline
	 ${\mbf e_{i}}$   &  ${\mbf  e_{j}}$ &   $-(E_{i,m+j}+E_{j,m+i})$      \\
	 ${\mbf e_{i}}$   &  ${\mbf  e_{m+j}}$ &   $-(E_{i,j}-E_{m+j,m+i})$      \\
	
	 ${\mbf e_{m+i}}$   &  ${\mbf  e_{m+j}}$ &   $E_{m+i,j}+E_{m+j,i}$      \\

	\hline 
\end{tabular}  
\end{center}
In the above table $i,j$ runs from $1$ to $m$.
\end{proof}
\begin{lemma}\label{l4.10}
	$\psi_B$ is a surjective map. 
	
\end{lemma}
\begin{proof}
From Lemma \ref{l3.1} we have $B=A^TJ_mA$ for some $A \in GL(N,\C)$. Also we know that $A\mf G_B A^{-1}=\mf G_{J_m}.$ Note that,
$\mf G_{J_m}=span\{ {\bf r}(J_m{\bf r})^T: {\bf r }\in \Z^N \}$
$= span\{{\bf r}(J_m{\bf r})^T: {\bf r }\in \C^N \}$
$= span\{{A\mbf r}(J_m{A\mbf r})^T: {\bf r }\in \C^N \}$, for some fixed $A \in GL(N,\C)$. Now consider ${\bf r}(B{\bf r})^T={\bf r}(A^TJ_mA{\bf r})^T ={\bf r}{\bf r}^TA^TJ_m^TA={\bf r}(J_mA{\bf r})^TA= A^{-1}(A{\bf r})(J_mA{\bf r})^TA. $ Therefore we have $span\{ {\bf r}(B{\bf r})^T: {\bf r} \in\Z^N \}= A^{-1}span\{A{\bf r}(J_mA{\bf r})^T : {\bf r }\in \Z^N\}A = A^{-1}\mf G_{J_m}A=\mf G_B.$ This completes the proof. 	
\end{proof}
\begin{lemma}\label{l4.11}
 $ker\psi_B$ is central ideal modulo $I_3$.
\end{lemma}

\begin{proof}
Since kernel of a Lie algebra homomorphism is an ideal, so  $ker\psi_B$ is an ideal. Moreover, it is easy to see that $I_3 \subseteq ker\psi_B$. Now let $X=\dis{\sum a_{\bf r}T({\bf r})} \in ker \psi_B$. Then $\dis{\sum a_{\bf r}{\bf r}(B{\bf r})^T}=0.$ Hence for any ${\bf k}=(k_1, \dots , k_N)$ we have, $\dis{\sum a_{\bf r}{\bf r}(B{\bf r})^T}{\bf k}=\dis{\sum a_{\bf r}(B{\bf r}|{\bf k}){\bf r}}=0$. This implies that $\dis{\sum a_{\bf r}(B{\bf r}|{\bf k}){r_i}}=0,$ for $1\leq i \leq N$. Now consider\\

$[T({\bf k}),X]= \dis{\sum a_{\bf r}[T({\bf k})},T({\bf r})]= \dis{\sum a_{\bf r}} (B{\bf k}|{\bf r})[T({\bf r +\bf k})-T({\bf k })-T({\bf r})]$ \\

$
=\dis{\sum a_{\bf r}} (B{\bf k}|{\bf r})T_2(\mbf 0,{\bf k, \bf r})=\dis{\sum a_{\bf r}} (B{\bf k}|{\bf r})\dis{\sum_{i}^{N}}r_iT_2{(\mbf 0,\mbf k,}$ $ {{\mbf e_i}})=0,$ as $\dis{\sum a_{\bf r}(B{\bf r}|{\bf k}){r_i}}=0,$ for $1\leq i \leq N$. Note that the 4th equality follows from Lemma \ref{l4.5}(3) and the fact that $I_3 \subseteq ker\psi_B$.

\end{proof}
Note that dimension of $ker \psi_B /I_3 =2m^2+m$, which follows from Lemma \ref{l4.8} and the fact that dimension of $\mf {sp}_{2m}=2m^2+m$. 
Now accumulating the results of Proposition \ref{p4.1} and Lemmas \ref{l4.9},\ref{l4.10},\ref{l4.11} we have proved the Theorem \ref{t4.1}.

\section{Level zero modules for SSEALA}
In this section we classify simple integrable $\tau_B$-modules of level zero and with finite dimensional weight spaces with respect to $\til {\mf h}$. We have $\tau_B= \tau_B^-  \op \tau_B^0 \op \tau_B^+$, where\\
$\tau_B^+=span\{X_\al({\bf r}):\al \in \De^+, {\bf r} \in \Z^N\}$,\\
		$\tau_B^-=span\{X_\al({\bf r}):\al \in \De^-, {\bf r} \in \Z^N\}$, and\\
			$\tau_B^0=span\{h({\bf r}),K(B{\bf r},{\bf r}),K({\bf u,0}),D({\bf u,0}),D(B{\bf{r}, r}):h \in \mf{h}, {\bf r} \in \Z^N , \bf u \in$ $ \C^N\}$.
			Throughout this section fix an irreducible integrable module $V$ for $\tau_B$ with finite dimensional weight spaces with respect to $\wtil{\mf h}$.
\begin{theorem}\label{t5.1}
	Let $V$ be an irreducible integrable module for $\tau_B$ with finite dimensional weight spaces with respect to $\wtil{\mf h}$. Suppose that all $K_i$, $ 1\leq i \leq N $ act trivially on $V$. Then there exists a weight vector $v_0$ such that ${\tau_B}^{+}. v_0=0 $.
\end{theorem}
\begin{proof}
First we apply the similar argument as in \cite{R04} to find a weight $\mu$ of $V$ such that $V_{\mu +\al}=0$ for all $\al \in \De^+$. We claim that $V_{\mu +\al +\de_{\bf m}}=0$ for all $\al \in \De^+, {\bf m}\in \Z^N$. If not, suppose $V_{\mu +\al +\de_{\bf r}}\neq 0$ for some $\be \in \De^+, {\bf r}\in \Z^N$.\\ 
{\bf Claim :} $V_{\la +\be +\de_{\bf s}}=0$ for all $\be \in \De^+, {\bf s}\in \Z^N$, $\la =\mu +\al+\de_{\bf r}$. Suppose this is also not true, then $V_{\la +\be +\de_{\bf s}} \neq 0$ for some $\be \in \De^+, {\bf s}\in \Z^N$. Since $\al , \be \in \De^+$ either $<\al +\be, \al> >0$ or $<\al +\be , \be> > 0$. Let  $<\al +\be, \al> >0$. Then we have $<\la +\be +\de_{\bf s},\al+\de_{\bf s}+\de_{\bf r}>=<\mu +\al+\be,\al> > 0$. Hence by Lemma \ref{l2.1} $V_{\mu +\be} \neq 0$, a contradiction. This completes the proof.	
\end{proof}
	Let $V^+=\{v \in V: \tau_B^+.v=0\}$. Then by Theorem \ref{t5.1} we have $V^+ \neq 0$. It is standard to prove that $V^+$ is an irreducible $\tau_B^0$-module (for a proof one can see the claim in the proof of Theorem 5.2 of \cite{RT}).  Note that $\mf h$ is central in $\tau_B^0$ and hence acts by a single linear functional (say $\bar \la$) on $V^+$. By Lemma \ref{l2.1} and Theorem \ref{t5.1}, $\bar \la$ is a dominant integral weight. Now fix a weight vector $v \in V^+$ of weight $\la$. Then all the weights of $V^+$ are lies in the set $ \{ \bar \la + \de_{\bf r}+\dis{\sum_{i=1}^{N}g_i\de_i}: {\bf r} \in \Z^N \}$ and some ${\bf g}=(g_1,\dots,g_N) \in \C^N$.
	\begin{remark}\label{r5.1}
		We assume that  $\mf g \ot \mcal A$ acts non-trivially on $V$. It is easy to see that this condition is equivalent to say that $\bar \la \neq 0$. If $\mf g \ot \mcal A$ acts trivially on $V$, then $V$ becomes an irreducible module for $\wtil H_B$ and we have no classification result for irreducible modules of $\wtil H_B$.
		
	\end{remark}
Before going to next proposition we define a uniformly bounded vector space. We call a $\Z^N$ graded vector space uniformly bounded if dimension of its components are bounded by a fixed natural number.

\begin{proposition}\label{p5.1}

		(1) Weight spaces of $V^+$ are uniformly bounded. \\
		(2) $\wtil{\mcal Z}$ acts trivially on $V^+$.

\end{proposition}
\begin{proof}
Let $\la$ be a weight of $V^+$. Then we know that all the weights are given by $\la_{\bf r}= \bar \la + \de_{\bf r}+\dis{\sum_{i=1}^{N}g_i\de_i},$ for all ${\bf r} \in \Z^N$. Define $\la^0=min\{\la(h): \la(h)>0\}$. Note that by Remark \ref{r5.1} we get $\la^0 \in \N.$ Now using similar argument as in [\cite{CP}, Lemma 3.1] with the fact that $K_i$ acting trivially we have, for any $\la_{\bf r}$ there exists a $\omega_r \in \widetilde W$ such that $\omega_{\bf r}(\la_{\bf r})=\la_{\bf s}$ and $0 \leq s_i < \la^0$ for $1 \leq i \leq N$. Therefore Lemma \ref{l2.1} suggests that dimensions of weight spaces of $V^+$ are bounded by max$\{dim V^+_{\la_{\bf s}}: 0\leq s_i < \la^0\},$ this proves (1).\\
Let $M=\dis{\bigoplus_{0 \leq s_i <\la^0 }V^+_{\la_{\bf s}}}$ and $\tau_N=\mf g\ot \mcal A\oplus \wtil { \mcal Z} \op  \dis{\bigoplus_{i=1}^{N}}\C d_i$. Note that $\tau_N$ is the quotient of $N$-variable toroidal Lie algebra. It is also clear from the irreducibility of $V$ that any $\tau_N$ submodule of $V$ intersecting $V^+$ non-trivially is generated by elements of $V^+$. Hence by Lemma \ref{l2.1} any such submodule is generated by a subset of $M$. Let $W_1 \supset W_2 \supset \dots,$ be a strictly decreasing sequence of $\tau_N$-submodules of $V$ intersecting $V^+$ non-trivially. As the dimension of $M$ is finite, this sequence must terminate. Therefore there exists a minimal submodule, say $V_{min}$ such that $V^+\cap V_{min} \neq 0$. This $V_{min}$ may not be irreducible. But we can find an unique irreducible quotient of $V_{min}$, say $\wtil V$. Then by \cite{RT} we have $\wtil {\mcal Z}$ acts trivially on $\wtil V$. This implies that $\wtil {\mcal Z}.(V^+\cap V_{min})=0$. Now it is easy to see that $W=\{v \in V^+:\wtil {\mcal Z}.v=0\}$ is a non zero $\tau_B^0$ submodule of $V^+,$ hence $W=V^+$.
\end{proof}

 Let $\alpha \in \De^+$ such that $\la(h_\al)\neq 0$. Then by results of \cite{CP2} it follows that $h_\al \ot t^{\bf r}$ acts trivially on $V^+$ for all ${\bf r} \neq \mbf 0$ if and only if $\la(h_\al)=0$. Therefore $\la(h_\al)\neq 0$ implies that $h_\al \ot t^{\bf r}$ acts non-trivially on $V^+$ for some ${\bf r}\neq \mbf 0$. Now it follows from Appendix that $h_\al \ot t^{\bf r}$ acts non-trivially on $V^+$ for all ${\bf r}\neq \mbf 0$. Now we have $V^+$ is an irreducible module for $\wtil H_B \ltimes (\mf h\ot \mcal A)$. Here we treat $\mf h \ot \mcal A$ as finitely many copies of $\mcal A$. In the appendix we have taken only one copy of $\mcal A$ but the arguments go through for finitely many copies of $\mcal A$ (we need to use the fact that $\mf h \ot  \mcal A$ is commutative). Thus from appendix we have $\la(h_\al) \neq 0$ on $V^+$. Hence by Theorem \ref{t6.1} we find non-zero scalars $\la_\al, \mu_\al$ such that on $V^+$ it satisfy the conditions\\
\begin{align}
	h_\al \ot t^{\bf r}.h_\al \ot t^{\bf s}=\la_\al h_\al \ot t^{\bf r+s}, \, \forall \, {\bf r,s, r+s} \notin \{ \mbf 0\},\\
		h_\al \ot t^{\bf r}.h_\al \ot t^{\bf -r}=\la(h_\al) \mu_\al, \forall \, {\bf r} \neq \mbf 0
\end{align}
and $ \la_\al^2=\mu_\al\la(h_\al).$
Then from Lemma 7.5 of \cite{RH} it follows that $\la_\al =\mu_\al=c=\la(h_\al)$ for all $\al \in \De^+$ such that $\la(h_\al)\neq 0$ [In Lemma 7.5 of \cite{RH} we get the equation $\dis{\sum_{i}\la_i(h_\al)a_i^s}=\la_\al,$ where LHS is independent of basis of $V^+$ and RHS is independent of s. But by the above equation both are independent of basis and s. Thus arguing as in Lemma 7.5 of \cite{RH} we see that $\mu_\al=\la_\al=\la(h_\al)=c$].\\
We know from Proposition \ref{p5.1} that $\wtil {\mcal Z}.V^+=0$. Now consider $W=\{v \in V: \wtil {\mcal Z}.v=0\}$ and observe that $W$ is a non-zero $\tau_B^0$ submodule of $V$. Therefore by irreducibility of $V$, $\wtil {\mcal Z} $ acts trivially on $V$. Hence we have a nice description of $V^+$ from Theorem \ref{t4.1}, given by $V^+ \simeq P\ot A $, where $P$ is a finite dimensional irreducible $\mf{sp}_{2m} $ module.
\begin{theorem}
	Let $V$ be an irreducible integrable level zero module for $\tau_B$ having finite dimensional weight spaces with respect to $\wtil h$ and $\mf g \ot \mcal A$ acts non-trivially on $V$. Then there exists a finite dimensional irreducible module $V(\mu)$ for $\mf g$ and a finite dimensional irreducible module $P$ for $\mf{sp}_{N}$ such that $V\simeq V(\mu)\ot P \ot \mcal A$. The action of $\tau_B$ on $ V(\mu)\ot P \ot \mcal A$ given by following:\\
	$X({\bf r}).(v\ot w\ot t^{\bf s})=X.v\ot w\ot t^{\bf r+s}$, for all $X \in \mf g$, $\bf r,s$$ \in \Z^N,$ $v \in V(\mu), w \in P$,\\
	 $\wtil {\mcal Z}$ acts trivially on $V$ and action of $\wtil {H_B}$ is given by (4.3) and (4.4).
\end{theorem}
\begin{proof}
	Note that the highest weight is dominant integral for $\mf g$, hence we can find $V(\mu)$. Then it is easy to see that $V(\mu)\ot P \ot \mcal A$ is an irreducible module for $\tau_B$. Observe that $\C v_\mu \ot P \ot \mcal A$ is the highest weight space for $\tau_B$-module $V(\mu)\ot P \ot \mcal A$ with respect to the considered triangular decomposition, where $v_\mu$ is the highest weight vector of $V(\mu)$. Now since highest weight space of $V$ and $V(\mu)\ot P \ot \mcal A$ are isomorphic as $\tau_B^0$-module. Hence $V \simeq V(\mu)\ot P \ot \mcal A$. For more clarity on induced modules and its unique irreducible quotient see subsection 6.1.
\end{proof}
\begin{remark}
Let us consider the matrix $J'=\begin{pmatrix}
	J_m &0_{2m,1} \\
	0_{1, 2m} & 0 \\ 
\end{pmatrix} $.
From Section 3 we see that $\tau_{J'} \simeq \tau_{\bar J}$. Though $J'$ looks much simpler than $\bar J$, it does not simplify the above arguments any further for $\bar J$.	
\end{remark}

\section{Non Zero level modules for KEALA}
In this section we classify non zero level simple integrable modules for the KEALA constructed in \cite{RH} and obtained in this paper when $B=\bar J$, i.e $\tau_{\bar J}$. In this section we assume that $N=2m+1$ and abbreviate $ \underline{\bf r}= \bar J{\bf r}=(r_{m+1}+r_N,\dots,r_{2m}+r_N,-r_1+r_{N},\dots,-r_m+r_N,-\dis{\sum_{i=1}^{2m}r_i})$. Now consider the triangular decomposition of $\tau_{\bar J}$ given by
\begin{align}
	\tau_{\bar J}^+=span \{\mf g^+ \ot t^{\bf m},\mf g \ot t^{\bf r},K(\underline{\bf s},{\bf s}),D(\underline{\bf s},{\bf s}):m_N=0,r_N,s_N>0, {\bf m,r,s} \in \Z^N, {\mbf s} \notin  G_{\bar J}\}\\
		\tau_{\bar J}^0=             span \{\mf h \ot t^{\bf m},K(\underline{\bf s},{\bf s}),D(\underline{\bf s},{\bf s}):m_N=0,s_N=0, {\bf m,s} \in \Z^N, {\bf s} \notin G_{\bar J}\}  \op \dis{\bigoplus_{1 \leq i \leq N} \C K_i}   \op \mcal D , \\
	\tau_{\bar J}^-=span \{\mf g^- \ot t^{\bf m},\mf g \ot t^{\bf r},K(\underline{\bf s},{\bf s}),D(\underline{\bf s},{\bf s}):m_N=0,r_N,s_N<0, {\bf m,r,s} \in \Z^N, {\bf s} \notin G_{\bar J}\},
	\end{align}
where $\mf g=\mf g^-\op \mf h \op \mf g^+$ is the usual triangular decomposition of $\mf g$. Let $V$ be an irreducible integrable non zero level (i.e not all $K_i$ acting trivially on $V$) module for $\tau_{\bar J}$ with finite dimensional weight spaces with respect to $\wtil{\mf h}$. Then by standard argument we can assume that after twisting the module by an automorphism (\ref{a2.5}) we have only $K_N$ acting non trivially and all other $K_i$ acting trivially on $V$. Now similar to the proof of Proposition 2.4 of \cite{RT} (also see Theorem 2.1 of \cite{RJ}) we have the following theorem. 
\begin{theorem}\label{tn6.1}
	Let $V$ be an irreducible integrable non zero level  module for $\tau_{\bar J}$ with finite dimensional weight spaces with respect to $\wtil{\mf h}$. Then there exists a weight $\mu$ of $V$ such that $\tau_{\bar J}^+.V_{\mu}=0$.
\end{theorem}
Let $W^+= \{v \in V:\tau_{\bar J}^+.v=0 \}$. Then by Theorem \ref{tn6.1}, $W^+ \neq 0$. By standard argument it follows that $W^+$ is an irreducible module for $\tau_{\bar J}^0$ with finite dimensional weight spaces. Let $\la$ be a weight of $W^+$. Then by irreducibility of $W^+$ we get weights of $W^+$ lies in the set  $ \{ \bar \la + \de_{\bf r}+\dis{\sum_{i=1}^{N}g_i\de_i}: {\bf r} \in \Z^{2m} \}$ for some ${\bf g}=(g_1,\dots,g_N) \in \C^N$ and $\bar \la =\la|_{\mf h}$.
\begin{lemma}\label{essential}
	The $t^{\bf m}K_N$ with $\mbf m=(m_1, \cdots , m_N) \in \Z^N \setminus \{ \mbf 0\}$, $m_N=0$ acts trivially on $W^+$.
\end{lemma}
\begin{proof}
	  First note that, when $r_N=0$ and $\dis{\sum_{i=1}^{N-1}r_i}=0$, then $K(\mbf e_N,{\bf r}) =0$ in $\tau_{\bar J}$. We have for $\bf r,s \in $ $\Z^N \setminus\{\mbf 0\}$ with  $r_N=s_N=0$,
	$$[D({\underline{\mbf s}},{\bf s}),K({\mbf e_N},{\bf r})]=(\underline{\mbf s}|{\mbf r})K({\mbf e_N},{\mbf r+ \mbf s}) +(\underline{\mbf s}|{\mbf e_N})K({\mbf s, \mbf r+ \mbf s}).$$
	Since $K({\bf s, r+s})$ acts trivially on $W^+$ by (2), hence we have $K(\mbf e_N,{\bf r+s})$ acts trivially on $W^+$, when $ (\underline{\bf s}|{\bf r}) \neq 0$. Now for  given ${\bf m} \in \Z^N$ with $m_N=0$ and  $\dis{\sum_{i=1}^{N-1}m_i}\neq 0$, it is easy to find ${\bf r,s} \in \Z^{N}$ such that  $r_N=s_N=0,$  $ (\underline{\bf s}|{\bf r}) \neq 0$, $\dis{\sum_{i=1}^{N-1}r_i}=0$ and ${\bf r+s}={\bf m}$. For example, let $a \in \Z \setminus \{0\}$, $\mbf r=a \dis{\sum_{i=1}^m}(\mbf e_i -\mbf e_{i+m})$ and set $\mbf s= \mbf m-\mbf r$. Then $(\underline{\mbf s}| \mbf r)= -a \dis{\sum_{i=1}^{N-1}}m_i \neq 0$. This completes the proof. 
\end{proof}
\begin{remark}\label{rn6.1}
	We assume that $\bar \la \neq 0$. Otherwise using $\mf{sl}_2$ theory and integrability of $V$ we can prove that $\mf h \ot \mcal A_{2m}$ acts trivially on $W^+$. Moreover, for any ${\bf r,s} \in \Z^N$ with $s_N=r_N=0$, we have $[h({\bf s}),h({\bf r})]=<h,h>K( {\bf r}, {\bf r+s}) $ acts trivially on $W^+$. Now consider $[D({\underline{\bf r}},{\bf r}),$ $K({\mbf e_N},{\bf s})]=(\underline{\bf r}|{\bf s})$ $K({\mbf e_N},{\bf r+s}) +(\underline{\bf r}| $${\mbf e_N})K({\bf r, r+s})$, for $r_N=s_N=0$. This implies that $K({\mbf e_N},{\bf r+s})$ acts trivially on $W^+$ for all $r_N=s_N=0$, when $(\underline{\bf r}|{\bf s})\neq 0$, with the help of Lemma 6.1. Now for any ${\bf k}\in \Z^N \setminus \{\mbf 0\}$ with $k_N=0$, it is easy to find ${\bf r, s} \in \Z^N$ such that $r_N=s_N=0$, $(\underline{\bf r}|{\bf s})\neq 0$  and ${\bf r+s =k}$. This proves that $W^+$ is an irreducible module for  $span\{D(\underline{\bf s},{\bf s}):s_N=0, {\bf s} \in \Z^N, {\bf s} \notin G_{\bar J}\} \op \mcal D$. We will prove (see Lemma \ref{ln6.1} below) that $span\{D(\underline{\bf s},{\bf s}):s_N=0, {\bf s} \in \Z^N, {\bf s} \notin G_{\bar J}\} $ is isomorphic to $H_{N-1}$ (see Example 2.3 for the construction of $H_{N-1}$) as a Lie algebra. Hence $W^+$ is an irreducible module for the Lie algebra $H_{N-1} \op \dis{\bigoplus_{i=1}^{N-1}}\C d_i$, as $d_N$ is central in $\tau_{\bar J}^0.$ But we don't have any classification theorem for $H_{N} \op \mcal D$ irreducible modules.
\end{remark} 

\begin{proposition}\label{pn6.1}

	(1) Weight spaces of $W^+$ are uniformly bounded. \\
	(2) $span \{t^{\mbf s}K_i :{\bf s }\in \Z^{N-1}, 1\leq i \leq N-1 \} \cap \wtil {\mcal Z}$ acts trivially on $W^+$.\\

\end{proposition}
\begin{proof}
	 By Remark \ref{rn6.1}, there exits $h \in \mf h$ such that $\bar{\la}(h) \neq 0$. Now by the fact that $\bar \la$ is dominant integral we have $\la^0=min\{\la(h): \la(h)>0\} > 0$. Now we can proceed similarly like Proposition \ref{p5.1}(1) to prove (1). To prove (2) we consider $M=\dis{\bigoplus_{\substack{ {\bf s}\in  \Z^{N-1}\\0 \leq s_i <\la^0 }}W^+_{\la_{\bf s}}}$, where $\la_{\bf s}$ are the weights of $W^+$. Also let $\tau_{N-1}=\mf g \ot \mcal A_{N-1} \op span \{t^{\bf s}K_i :{\bf s }\in \Z^{N-1}, 1\leq i \leq N-1 \} \cap \wtil {\mcal Z} \op \dis{\bigoplus_{i=1}^{N-1}}\C d_i $ be the quotient of $N-1$ variable toroidal Lie algebra. Now the proof follows similarly to \ref{p5.1}(2).

\end{proof}
\begin{lemma}\label{ln6.1}
Let	$\mcal S=span \{ D(\underline{\bf s},{\bf s}): s_N=0, {\bf s} \notin G_{\bar J}\}$. Then $\mcal S \simeq H_{N-1}$.
\end{lemma}
\begin{proof}
	Define a map $\eta : \mcal S \to H_{N-1}$ by $\eta ( D(\underline{\bf s},{\bf s})) =  D({J_m\bf s},{\bf s}) $. We check that it is a Lie algebra homomorphism. Note that \\
	$$\eta([D(\un{\bf r},{\bf r}),D(\un{\bf s},{\bf s})])=(\un{\bf r}|{\bf s})\eta( D(\un{\bf r+s},{\bf r+s}))=(\un{\bf r}|{\bf s}) D(J_m({\bf r+s}),{\bf r+s})$$
	$$ =({J_m\bf r}|{\bf s})D(J_m({\bf r+s}),{\bf r+s})= [(D({J_m\bf r},{\bf r})),(D({J_m\bf s},{\bf s}))]=[\eta(D(\un{\bf r},{\bf r})),\eta(D(\un{\bf s},{\bf s}))],$$
 observe that $(\un{\bf r}|{\bf s})=({J_m\bf r},{\bf s})$ due to the fact that $s_N=0$. It is easy to see that $\eta $ is bijective.
\end{proof}
Therefore from Proposition \ref{pn6.1} and Lemma \ref{ln6.1} we have that $W^+$ is an irreducible module for $\mf h \ot \mcal A_{N-1} \op H_{N-1} \op \dis{\bigoplus_{i=1}^{N-1}}\C d_i $ with uniformly bounded weight spaces, as $K_N,d_N$ are central in $\tau_{\bar J}^0$. Now we proceed similarly like SSEALA case in Section 5 and conclude that $W^+$ is an irreducible module for $\mcal A_{N-1}\rtimes ( H_{N-1} \op \dis{\bigoplus_{i=1}^{N-1}}\C d_i)$. Since $H_{N-1}=H_{J_m}$, therefore by Theorem \ref{t4.1} (also see \cite{TH} Theorem 5.2 ) we have $W^+ \simeq P'\ot \mcal A_{N-1} $, where $P'$ is a finite dimensional irreducible $\mf{sp}_{N-1} $ module.\\
\subsection{Construction of induced module and its irreducible quotient for $\tau_{\bar J}$}Let $W$ be an irreducible module for $\tau_{\bar J}^0$. Now consider the Verma module
  $$M(W)=U(\tau_{\bar J}) \bigotimes_{\tau_{\bar J}^+\op \tau_{\bar J}^0} W,  $$ where $\tau_{\bar J}^+$ acts trivially on $W$. Let $L(W)$ be the unique irreducible quotient of $M(W)$. Moreover it is easy to see that if $W_1 \simeq W_2$ as $\tau_{\bar J}^0$ then $L(W_1) \simeq L(W_2)$ as $\tau_{\bar J}$ module. Now since $W^+ \simeq  P'\ot \mcal A_{N-1} $ as $\tau_{\bar J}^0$ module, hence we have the following theorem.
  \begin{theorem}
  	Let $V$ be an irreducible integrable non zero level module for $\tau_{\bar J}$ with finite dimensional weight spaces with respect to $\wtil {\mf h}$ and $\mf h$ acts non-trivially on $V$. Then upto a twist of an automorphism $V \simeq L(P'\ot \mcal A_{N-1})$, where $P'$ is a finite dimensional irreducible $\mf{sp}_{N-1} $ module.
  \end{theorem}

\section{appendix}

In this section we prove two technical result which we have used in last three sections. We follow the approach of \cite{GL,
	RH,GLZ} to prove Theorem \ref{t6.1}. For convenient we mention some notations here as in section 4. Recall the action of $ {H_B} $ on $\mcal A$ given by \begin{align}\label{a6.1}
	D(B\mbf r,\mbf r).t^{\mbf s}=(B{\bf r| s})t^{\bf r+s},
\end{align} 
for all $\bf r, s$$ \in \Z^N$. 
Let $G= H_B \ltimes \mcal A$ be the semi direct product of $H_B $ and $\mcal A$. Clearly $G = \dis{\bigo_{\mbf r \in \Z^N}}G_{\bf r}$ is a $\Z^N$ graded Lie algebra. Let $G'=\dis{\bigo_{\mbf r \in \Z^N \setminus \{\mbf 0\}}}G_{\bf r}$, then $G'$ is an ideal of $G$. Also let $\mcal A'=\dis{\bigo_{\mbf r \in \Z^N \setminus\{\mbf 0\}}}\C t^{\bf r}$. Let $W$ be an irreducible module for $G$ which satisfy the following properties.
\begin{enumerate}
	\item $W = \dis{\bigoplus_{\mbf r \in \Z^N}W_{\bf r}}$ and $G_{\bf r}.W_{\bf s}\subseteq W_{\bf r+s}$.
	\item $W$ is uniformly bounded $G$ module, i.e there exists a natural number $M$ such that dim $W_{\bf r} \leq M$ for all ${\bf r} \in \Z^N$.
\end{enumerate}
\begin{theorem}\label{t6.1}
	Let $\mcal A'$ acts non-trivially on $W$ and $t^{\bf 0}$ acts as a non-zero scalar $c$ on $W$. Then there exists non-zero scalars $\la_{\bf r, s}$ such that $t^{\bf r}t^{\bf s}=\la_{\bf r,s}t^{\bf r+s}$ on $W$. Moreover\\
	\begin{enumerate}
		\item $\la_{\bf r,s}=\la $ for all ${\bf r,s }\neq \mbf 0,{\bf r+s} \neq \mbf 0.$
		\item $\la_{\bf r, -r}=\mu $ for all ${\bf r} \neq \mbf 0$.
		\item $\la_{\mbf 0, \bf r}=c$ and $\la^2=\mu c$.
	\end{enumerate}
	
\end{theorem}
We prove this theorem using several lemmas. The next two lemmas follows from Lemma 3.1 and Lemma 3.2 of \cite{GLZ}.
\begin{lemma}

	 Let $W = \dis{\bigoplus_{\mbf r \in \Z^N}W_{\bf r}}$ be an irreducible uniformly bounded module for $\mcal A $. Then dim $W_{\bf r} \leq 1$ for all $\mbf r \in \Z^N$.

\end{lemma}
\begin{lemma}
		Let $g \in U(\mcal A)$ such that $g.v=0$ for some $v \in W$. Then $g$ is locally nilpotent on  $W$. 
	
\end{lemma}

\begin{lemma}\label{L6.3}
(1)	
Either each $t^{\mbf r}$, $\mbf r\neq 0$ acts injectively on $W$ or every $t^{\bf r}$ acts locally nilpotently on $W$.\\
(2) If $t^{\bf r}$ acts locally nilpotently for all $\mbf r \neq \mbf 0$, then $\mcal A'.W=0$.

\end{lemma}
\begin{proof}
	Let $t^{\bf r}$ acts injectively on $W$ for some ${\bf r} \neq \mbf 0$. Let ${\bf s} \in \Z^N$ such that $(B{\bf r}|{\bf s})\neq 0$ and $t^{\bf s}$ acts locally nilpotently on $W$. Then using the arguments similar to Claim 1 of [\cite{GL},Proposition 3.4 ] we have $t^{\bf s}$ acts nilpotently on $W$. Now following arguments of Claim 3 of [\cite{GL},Proposition 3.4 ] and the fact that $(B{\bf r}|{\bf s})\neq 0$ we get a contradiction. Hence $t^{\bf s}$ acts injectively on $W$ for all ${\bf s}$ satisfying $(B{\bf r}|{\bf s})\neq 0$. If $(B{\bf r}|{\bf s})=0$ then consider ${\mbf s_1} \in \Z^N$ such that $(B{\mbf r}|{\mbf s_1})\neq 0$ and $(B{\mbf s_1}|{\mbf s})\neq 0$. Hence by above we have action of $t^{\bf s}$ is injective on $W$. This completes the proof of 1. Proof of 2 follows just by changing $\bar {\bf r}$ to $B\mbf r$ in the proof of Proposition 9.5 of \cite{RH}. 
\end{proof}
{\bf Proof of Theorem \ref{t6.1} :} The proof runs parallel to Theorem 9.1, \cite{RH}. From Lemma \ref{L6.3} it follows that $t^{\bf r}$ acts injectively on $W$ for all $\mbf r \neq \mbf 0$. This implies that dim $W_{\bf r}$ = dim $W_{\bf s}$ for all $\bf r,s \in $$\Z^N$. Now by defining $T_{\bf r, \bf s}=t^{\bf r}t^{\bf s}-\la_{\bf r,s}t^{\bf r+s}$ and following the proof of Theorem 9.1, \cite{RH} we have 
\begin{align}\label{a6.2}
	(B{\bf l}|{\bf s})\la_{\bf{ r,s+l}}+(B{\bf l}|{\bf r})\la_{\bf s,r+l}-(B{\bf l}|{\bf r+ s})\la_{\bf r,s}=0,
\end{align}
for all ${\bf l,r,s}\in \Z^N \setminus \{\mbf 0\}.$ Now we prove some lemmas to complete the proof of Theorem \ref{t6.1}.
\begin{lemma}\label{L6.4}
	Let $(B{\bf l}|{\bf s})$$\neq 0$. Then we have
	\begin{enumerate}
		\item $\la_{\bf s+l,s}=\la_{\bf s,s}$.
			\item $\la_{\bf s,l}=\la_{\bf s,s}=\la_{\bf l,l}$.
		\item $\la_{\bf s,s}=\la_{\bf s+l,s+l}=\la_{\bf l,l}$.
			\item $\la_{{\bf s+l},j{\bf l}}=\la_{\bf s,s}=\la_{\bf l,l}$ for all $j \in \Z \setminus \{0\}$.
			\item $\la_{{\bf s},j{\bf s +l}}=\la_{\bf s,s}=\la_{\bf l,l}$ for all $j \in \Z $.
	\end{enumerate}
\end{lemma}
\begin{proof}
Proof of this lemma runs parallelly with Lemma 9.6 of \cite{RH}, just note that $(B{\bf s}|{\bf s})$$=0$.
\end{proof}

\begin{lemma}\label{L6.6}
	Let ${\bf r,s }\in \Z^N \setminus \{\mbf 0\} $ be such that $(B{\bf r}|{\bf s}) =0$ and ${\bf r+s} \neq \mbf 0$. Then we have the following.
	\begin{enumerate}
			\item $\la_{\bf r,r}=\la_{\bf s,s}=\la_{\bf r+s,r+s}$.
				\item $\la_{\bf s+r,s}=\la_{\bf s,s}=\la_{\bf r,r}=\la_{\bf r,s+r}$.
				\item $\la_{\bf r,s}=\la_{\bf s,s}=\la_{\bf r,r}$.	
				\item $\la_{{\bf s+r,}j{\bf r}}=\la_{\bf s,s}=\la_{\bf r,r}$, for all $j \in \Z \setminus \{0\}.$
					\item $\la_{{\bf s,}j{\bf s+r}}=\la_{\bf s,s}=\la_{\bf r,r}$, when $j{\bf s} +{\bf r} \neq \mbf 0$, $j \in \Z$.
	\end{enumerate}

\end{lemma}
\begin{proof}
 First note that (4) and (5) follows immediately from (2) and (3). Therefore we prove (1,2,3). Let $U_{\bf r}$ be denote the orthogonal complement of $B{\bf r}$ for all ${\bf r} \in \Z^N$.  Since $B$ is non degenerate $\Z^N \setminus (U_{\bf r}\cup U_{\bf s} \cup U_{\bf r+s}) \cap \Z^N $ is a proper subset of $\Z^N$. Hence there exists a ${\bf k} \in \Z^N$ such that $ (B{\bf r}|{\bf k}),(B{\bf s}|{\bf k}), (B({\bf r+s})|{\bf k}) \in \C \setminus \{0\}$.\\ Now from Lemma \ref{L6.4} we have $\la_{\bf r,r}=\la_{\bf s,s}=\la_{\bf r+s,r+s}=\la_{\bf k,k}$. Notice that $(B{\bf r}|{\bf s+k})\neq 0 \neq (B{\bf s}|{\bf r+k})$. Hence from Lemma \ref{L6.4} and the equality $\la_{\bf r,r}=\la_{\bf s,s}$  we have $\la_{\bf r,r}=\la_{\bf r,s+k}= \la_{\bf s,s}=\la_{\bf s,r+k}$. Now in equation (\ref{a6.2}) substituting ${\bf l}={\bf k}$ we have $\la_{\bf r,s+k}=\la_{\bf r,s}.$ Thus we have $\la_{\bf r,r}=\la_{\bf r,s}= \la_{\bf s,s}$. Now put ${\bf s}={\bf r+s}$ in the last equality, then we have $\la_{\bf r,r}=\la_{\bf r,r+s}= \la_{\bf r+s,r+s}$. Similarly, $\la_{\bf s,s}=\la_{\bf s,r+s}= \la_{\bf r+s,r+s}$. This completes the proof.
\end{proof}

The next lemma completes the proof of Theorem \ref{t6.1} and it follows from verbatim same proof of Lemma 9.9 and Lemma 9.10 of \cite{RH} using Lemma \ref{L6.4} and Lemma \ref{L6.6}.
\begin{lemma}

		(1) $\la_{j{\mbf s},j{\mbf s}}=\la_{\mbf s,\mbf s}$ for all $j \in \Z \setminus \{0\}$.\\
		(2) $\la_{j{\bf s},p{\bf s}}=\la_{\bf s,s}$ for all $ j+p \in \Z \setminus \{0\}$.\\
		(3) $\la_{{\bf r},{\bf s}}=\la$ for all ${\bf r+s} \neq \mbf 0$ and ${\bf r,s} \neq \mbf 0.$\\
	($4$)	$\la_{\mbf 0,\bf s}=c,$ for all $\mbf s $ $\in \Z^N$.\\
	(5) $\la_{\bf r, -r}=\mu$ for all $\bf r$ $ \neq  \mbf 0$.\\
	(6) $c\mu=\la^2$.
	
	 \end{lemma}
 
{ \bf Proof of Lemma 4.4(3)}, i.e  $[I_p,I_q]\subseteq I_{p+q-2}$ for all $p+q\geq 3$ and  $[I_1,I_1] \subseteq   I_1.$\\
Let $I \subseteq \{1,2,\dots,p\}$ and $J \subseteq \{1,2,\dots,q\}$. Let ${\mbf s_I}=\dis{\sum_{i \in I}{\mbf s_i}}$ and ${\mbf r_J}=\dis{\sum_{i \in J}{\mbf r_i}}$. Also set ${\mbf s_\Phi}=0 $ and ${\mbf r_\Phi}=0$. Then \\

$T_p({\mbf k, \mbf s_1,\dots, \mbf s_p})=\dis{\sum_{0\leq a \leq p}\sum_{|I|=a}(-1)^aT({\mbf k+ \mbf s_I})}$ and \\

$T_q({\mbf l,\mbf r_1,\dots,\mbf r_q})=\dis{\sum_{0\leq b \leq q}\sum_{|J|=b}(-1)^bT({\mbf l+ \mbf r_J})}.$ \\

Therefore, $[T_p({\mbf k,\mbf s_1,\dots,\mbf s_p}),T_q({\mbf l,\mbf r_1,\dots,\mbf r_q})]$\\

$=\dis{\sum_{\substack{0\leq a \leq p\\0 \leq b \leq q}}\sum_{\substack{|I|=a\\|J|=b}}}(-1)^{a+b} [T({\mbf k+\mbf s_I}),T({\mbf l+ \mbf r_J})]$\\
$=-\dis{\sum_{\substack{0\leq a \leq p\\0 \leq b \leq q}}\sum_{\substack{|I|=a\\|J|=b}}}(-1)^{a+b}(B({\mbf k+\mbf s_I})|{\mbf l+ \mbf r_J})(T({\mbf k+ \mbf s_I})+T({\mbf l+\mbf r_J})-T({\mbf k+\mbf l+\mbf s_I+ \mbf r_J})) .$\\
{\bf Claim 1:} $\dis{\sum_{\substack{0\leq a \leq p\\0 \leq b \leq q}}\sum_{\substack{|I|=a\\|J|=b}}}(-1)^{a+b}(B({\mbf k+\mbf s_I})|{\mbf l+ \mbf r_J})T({\mbf l+ \mbf r_J})=0.$ \\
Fix some $J$ such that $|J|=b$. Now to prove this claim it is sufficient to prove that $\dis{\sum_{0\leq a \leq p}\sum_{|I|=a}}(-1)^a(B({\mbf k+\mbf s_I})|{\mbf l+\mbf r_J})=0, $ which means it is sufficient to prove that\\ $\dis{\sum_{0\leq a \leq p}\sum_{|I|=a}}(-1)^a({\mbf k+ \mbf s_I})=0$. But this sum is equal to \\
$\dis{\sum_{0\leq a \leq p}}(-1)^a{{p}\choose {a}}{\bf k}+\dis{\sum_{0\leq a \leq p}\sum_{|I|=a}}(-1)^a{\mbf s_I}$, and this is easy to check to be 0.\\
{\bf Claim 2:} $\dis{\sum_{\substack{0\leq a \leq p\\0 \leq b \leq q}}\sum_{\substack{|I|=a\\|J|=b}}}(-1)^{a+b}(B({\mbf k+ \mbf s_I})|{\mbf l+ \mbf r_J})T({\mbf k+\mbf s_I})=0.$ This follows similarly like Claim 1.\\
Now we evaluate $\dis{\sum_{\substack{0\leq a \leq p\\0 \leq b \leq q}}\sum_{\substack{|I|=a\\|J|=b}}}(-1)^{a+b}(B({\mbf k+ \mbf s_I})|{\mbf l+ \mbf r_J})T({\mbf k+\mbf l+\mbf s_I+\mbf r_J}) .$ For that we write the above expression as the sum of four terms given by
\begin{align}
	\dis{\sum_{\substack{0\leq a \leq p\\0 \leq b \leq q}}\sum_{\substack{|I|=a\\|J|=b}}}(-1)^{a+b}(B{\mbf k}|{\mbf l})T({\mbf k+\mbf l+\mbf s_I+\mbf r_J})\\
	\dis{\sum_{\substack{0\leq a \leq p\\0 \leq b \leq q}}\sum_{\substack{|I|=a\\|J|=b}}}(-1)^{a+b}(B{\mbf s_I}|{\mbf l})T({\mbf k+ \mbf l+\mbf s_I+\mbf r_J})\\
	\dis{\sum_{\substack{0\leq a \leq p\\0 \leq b \leq q}}\sum_{\substack{|I|=a\\|J|=b}}}(-1)^{a+b}(B{\mbf k}|{\mbf r_J})T({\mbf k+\mbf l+\mbf s_I+\mbf r_J})\\
	\dis{\sum_{\substack{0\leq a \leq p\\0 \leq b \leq q}}\sum_{\substack{|I|=a\\|J|=b}}}(-1)^{a+b}(B{\mbf s_I}|{\mbf r_J})T({\mbf k+\mbf l+\mbf s_I+\mbf r_J}).
\end{align}
Now consider (7.3) which is clearly equal to $(B{\bf k}|{\bf l})T_{p+q}({\mbf k+\mbf l,\mbf s_1,\dots,\mbf s_p,\mbf r_1,\dots,\mbf r_q}).$ Now consider the expression (7.5), to simplify this we fix $i \in \{1,2,\dots,q\}$. Then consider the expression\\
$\dis{\sum_{\substack{0\leq a \leq p\\1 \leq b \leq q}}\sum_{\substack{|I|=a\\|J|=b}}}(-1)^{a+b}(B{\mbf k}|{\mbf r_i})T({\mbf k+\mbf l+\mbf s_I+\mbf r_J})$ (Note that $b=0$ not occur here). This is equal to $(B{\mbf k}|{\mbf r_i})T_{p+q-1}({\mbf k+\mbf l+\mbf r_i,\mbf s_1,\dots,\mbf s_p,\mbf r_1,\dots,\hat{ \mbf r}_i,\dots,\mbf r_q}).$ Hence (7.5) is equal to $-\dis{\sum_{i=1}^{q}}(B{\mbf k}|{\mbf r_i})T_{p+q-1}({\mbf k+\mbf l+\mbf r_i,\mbf s_1,\dots, \mbf s_p,\mbf r_1,\dots,\hat{\mbf r}_i,\dots,\mbf r_q}).$\\
 In the similar manner (7.4) is equal to $-\dis{\sum_{i=1}^{p}}(B{\mbf s_i}|{\mbf l})T_{p+q-1}({\mbf k+\mbf l+\mbf s_i,\mbf s_1,\dots,\hat{\mbf s}_i,\dots, \mbf s_p,\mbf r_1,\dots,\mbf r_q}).$ Now fix $i \in \{1,2, \dots , p\}$, $j \in \{1,2, \dots, q\}$ and consider the expression \\
 $ \dis{\sum_{\substack{1\leq a \leq p\\1 \leq b \leq q}}\sum_{\substack{|I|=a\\|J|=b\\i \in I,j \in J}}}(-1)^{a+b}(B{\mbf s_i}|{\mbf r_j})T({\mbf k+\mbf l+ \mbf s_I+\mbf r_J})$\\ $=(B{\mbf s_i}|{\mbf r_j})T_{p+q-2}({\mbf k+\mbf l+\mbf s_i+\mbf r_j,\mbf s_1,\dots,\hat{\mbf s}_i,\dots, \mbf s_p,  \mbf r_1,\dots,\hat{\mbf r}_j,\dots,\mbf r_q}).$ Therefore (7.6) is equal to 
 $\dis{\sum_{\substack{1\leq i \leq p\\1 \leq j \leq q}}}(B{\mbf s_i}|{\mbf r_j})T_{p+q-2}({\mbf k+\mbf l+\mbf s_i+\mbf r_j,\mbf s_1,\dots,\hat{\mbf s}_i,\dots, \mbf s_p,\mbf r_1,\dots,\hat{\mbf r}_j,\dots,\mbf r_q}).$ This argument fails when $p=q=1$. In fact in that case (7.6) is equal to $(B{\mbf s_1}|{\mbf r_1})T({\mbf k+\mbf l+\mbf s_1+\mbf r_1})$.\\
 
  Therefore $[T_p({\mbf k, \mbf s_1,\dots,  \mbf s_p}),T_q({\mbf l, \mbf r_1,\dots, \mbf r_q})]$\\ =
$ \begin{cases}
 	-(B{\mbf k}|{\mbf l})T_{p+q}({\mbf k+ \mbf l, \mbf s_1,\dots, \mbf s_p, \mbf r_1,\dots, \mbf r_q}) \\+\dis{\sum_{i=1}^{q}}(B{\mbf k}|{\mbf r_i})T_{p+q-1}({\mbf k+ \mbf l+ \mbf r_i, \mbf s_1,\dots, \mbf s_p, \mbf r_1,\dots,\hat{ \mbf r}_i,\dots, \mbf r_q})\\
 	+\dis{\sum_{i=1}^{p}}(B{\mbf s_i}|{\mbf l})T_{p+q-1}({\mbf k+  \mbf l+ \mbf s_i, \mbf s_1,\dots,\hat{ \mbf s}_i,\dots,  \mbf s_p, \mbf r_1,\dots, \mbf r_q})\\
 	- \dis{\sum_{\substack{1\leq i \leq p\\1 \leq j \leq q}}}(B{\mbf s_i}|{\mbf r_j})T_{p+q-2}({\mbf k+ \mbf l+ \mbf s_i+ \mbf r_j, \mbf s_1,\dots,\hat{ \mbf s}_i,\dots,  \mbf s_p, \mbf r_1,\dots,\hat{ \mbf r}_j,\dots, \mbf r_q}), when  \, \, (p,q)\neq (1,1) \\
 -(B{\mbf k}|{\mbf l})T_2({\mbf k+ \mbf l, \mbf s_1, \mbf r_1}) +(B{\mbf k}|{\mbf r_1})T_1({\mbf k+  \mbf l+ \mbf r_1, \mbf s_1})+(B{\mbf s_1}|{\mbf l})T_1({\mbf k+ \mbf l+ \mbf s_1, \mbf r_1}) \\
  -(B{\mbf s_1}|{\mbf r_1})T_1({\mbf k+ \mbf l+ \mbf s_1+ \mbf r_1}), \,\, when \, (p,q)=(1,1).
 \end{cases} $
 
 
  This completes the proof.\\

{\bf Acknowledgments:} 
We would like to thank anonymous referee for helpful suggestions to improve the paper.
  For this project the second author is supported by funds of the National Natural Science Foundation of China (Grant No. 12071136) and Science and Technology Commission of Shanghai Municipality (Grant No. 22DZ2229014).


\begin{thebibliography}{99}
\bibitem[AABGP]{AABGP}  B. Allison, S. Azam, S. Berman, Y. Gao, A. Pianzola (1997): Extended affine Lie algebras and their root
systems, Mem. Am. Math. Soc. 126 (603), x+122 pp.
\bibitem[AG]{AG} B. Allison, Y. Gao  (2001): The root system and the core of an extended affine Lie algebra, Sel. Math. New
Ser. 7 (2) 149–212.

\bibitem[ABFP]{ABFP} B. Allison, S. Berman, J. Faulkner, A. Pianzola (2009): Multiloop realization of extended affine Lie algebras
and Lie tori, Trans. Am. Math. Soc. 361 (9)  4807–4842.
\bibitem[BT]{BT} Billig, Yuly, Talboom, John (2018): Classification of category J modules for divergence zero vector fields on a torus, J. Algebra 500 (2018), 498–516.

\bibitem[CLT1]{CLT1}  F. Chen, Z. Li, S. Tan (2019): Integrable representations for toroidal extended affine Lie algebras, J. Algebra,519,  228–252.
\bibitem[CLT2]{CLT2} F. Chen, Z. Li, S. Tan (2021): Classification of integrable representations for toroidal extended affine Lie
algebras, J. Algebra 574 , 1–37.
\bibitem[CP1]{CP} ] V. Chari, J. Greenstein (2008):  Graded level zero integrable representations of affine Lie algebras, Trans.
Am. Math. Soc. 360 (6)  2923–2940.
\bibitem[CP2]{CP2} Vyjayanthi Chari and Andrew Pressley (2001): Weyl modules for classical and quantum
affine algebras. Represent. Theory, 5:191–223.

\bibitem[DKK]{DKK}  Dragomir Z. Dokovic, Konstanze Rietsch, Kaiming Zhao (2007):  Normal forms for orthogonal similarity classes
of skew-symmetric matrices, Journal of Algebra 308, 686–703.

\bibitem[FRK]{FRG} F.R. Gantmacher ( 1989): Theory of matrices, Vol 2, chelsea, New York.



\bibitem[GL]{GL} X. Guo, G. Liu (2019): Jet modules for the centerless Virasoro-like algebra, J. Algebra Appl. 18 (1) 
1950002.
\bibitem[GLZ]{GLZ} X. Guo, G. Liu, K. Zhao (2014): Irreducible Harish-Chandra modules over extended Witt algebras, Ark.
Mat. 52 (1)  99–112.
\bibitem[GSR]{GSR} G S Rinehart (1983): Differential forms on general commutative algebras, Trans. Amer. Math. Soc. 108, 195-222.
\bibitem[GW]{GW} Roe Goodman, Nolan R. Wallach, Symmetry, Representations, and Invariants, Graduate Texts in Mathematics 255, Springer, Cambridge University Press, 1998, third corrected printing 2003.

   
\bibitem[JH]{JH}
J.E. Humphreys (1972): { Introduction to Lie Algebras and Representations Theory}, Spinger , Berlin, Heidelberg, New york, .
\bibitem[MRY]{MRY} R.V. Moody, S. Eswara Rao, T. Yokonuma (1990): Toroidal Lie algebras and vertex representations, Geom.
Dedic. 35 (1–3)  283–307.
\bibitem[N1]{N1} E. Neher (2004): Lie Tori, C. R. Math. Acad. Sci. Soc. R. Can. 26 (3)  84–89.
\bibitem[N2]{N2} E. Neher (2004): Extended affine Lie algebras, C. R. Math. Acad. Sci. Soc. R. Can. 26 (3)  90–96.
\bibitem[PR]{PR} S. Pal, S. Eswara Rao (2021): Classification of level zero irreducible integrable modules for twisted full
toroidal Lie algebras, J. Algebra 578  1–29.

\bibitem[R1]{RH} S. Eswara Rao (2023): Hamiltonian extended affine Lie algebra and its representation theory.  {\it J. Algebra.} 68, 71--97.
\bibitem[R2]{RT} S.Eswara Rao (2004): Classification of irreducible integrable modules for toroidal Lie algebras with finite dimensional weight spaces.{\it J. Algebra }, 277, 318--348.
\bibitem[R3]{R05} S. Eswara Rao (2005): Irreducible representations for toroidal Lie-algebras, J. Pure Appl. Algebra 202 (1–3) 102–117.
\bibitem[R4]{R04}  S. Esawara Rao (2004): Partial classification of modules for Lie-algebra of diffeomorphisms of $d-$dimensional torus, J.Math.Phys, 45, no 8, 3322-3333.
\bibitem[RA]{RA} A.N. Rudakov (1974): Irreducible representations of infinite-dimensional Lie algebras of Cartan type, Izv.
Akad. Nauk SSSR, Ser. Mat. 38  835–866 (Russian).
\bibitem[RJ]{RJ} S.Eswara Rao, Cuipo Jiang (2005): Classification of irreducible integrable representations for the full toroidal Lie algebras. {\it J. Pure Appl. Algebra, } 200, No 1-2, 71--85.
\bibitem[RM1]{RM1} S.Eswara Rao, R.V.Moody (1994): vertex representation for n-toroidal Lie algebras and a generalization of the Virasoro algebra. {\it Comm. Math. Phys} 159, No 2, 239--264.
\bibitem[RSB]{RSB} S. Eswara Rao, S. Sharma, P. Batra (2020): Integrable modules for twisted toroidal extended affine Lie
algebras, J. Algebra 556  1057–1072.
\bibitem[RSM]{RSM} S. Eswara Rao, S. Sharma, S. Mukherjee (2021): Integrable modules for loop affine-Virasoro algebras,
Commun. Algebra 49 (12)  5500–5512.

\bibitem[RSS]{RSS}Ramos, E., Sah, C. H., and Shrock, R. E.(1990):  Algebra of diffeomorphisms of $N$-torus, {\it J. Math phys,} 31, 1805-1816,1990.
\bibitem[SP]{SP} Souvik Pal (2021): Integrable modules for graded Lie tori with finite-dimensional weight spaces, J. Pure Appl. Algebra 225, no. 9, Paper No. 106679, 21 pp.
\bibitem[SP1]{SP1}  Souvik Pal(2025): Classification of irreducible Harish-Chandra modules over full toroidal Lie algebras and higher-dimensional Virasoro algebras, Math. Res. Lett., Vol. 32, No. 4, 1197-1248, .


\bibitem[SP2]{SP2}  Souvik Pal (2025): Quasi-finite modules over extended affine Lie algebras, Math. Z. 311 , No. 3, Paper No. 63, 35 pp.
\bibitem[Sam P]{Sam P} Sam Perlis (1958): Theory of matrices, Addison-Wesley publishing company, inc, Third Printing.
\bibitem[T1]{TH} J. Talboom (2018): Category J modules for Hamiltonian vector fields on a torus, J. Lie Theory 28 (4), 903-914.
\bibitem[TB]{TB} S. Tantubay, P. Batra (2023): Classification of irreducible integrable modules for extended affine Lie algebras with center acting trivially, J. Lie Theory 33, no. 3, 747–762. 

\bibitem[VK]{VK} V.G. Kac (1990): Infinite-Dimensional Lie Algebras, third edition, Cambridge University Press, Cambridge.
\bibitem[Y1]{Y1}  Y. Billig (2006): A category of modules for the full toroidal Lie algebra, Int. Math. Res. Not.  68395.
\bibitem[YY]{YY} Y. Yoshii (2006): Lie tori–a simple characterization of extended affine Lie algebras, Publ. Res. Inst. Math.
Sci. 42 (3)  739–762.









\end{thebibliography}
\end{document}